\documentclass[reqno]{amsart}  	
\usepackage{geometry}                		
\usepackage{graphicx}				
\usepackage{todonotes}
\usepackage{amssymb}

\usepackage{amssymb}
\usepackage{amsthm}
\usepackage{graphicx}
\usepackage{mathtools}
\usepackage{amsmath,amsfonts}
\usepackage{epsfig}
\usepackage{comment}
\usepackage{color}
\definecolor{AliceBlue}{rgb}{0.94,0.97,1.00}
\definecolor{AntiqueWhite1}{rgb}{1.00,0.94,0.86}
\definecolor{AntiqueWhite2}{rgb}{0.93,0.87,0.80}
\definecolor{AntiqueWhite3}{rgb}{0.80,0.75,0.69}
\definecolor{AntiqueWhite4}{rgb}{0.55,0.51,0.47}
\definecolor{AntiqueWhite}{rgb}{0.98,0.92,0.84}
\definecolor{BlanchedAlmond}{rgb}{1.00,0.92,0.80}
\definecolor{BlueViolet}{rgb}{0.54,0.17,0.89}
\definecolor{CadetBlue1}{rgb}{0.60,0.96,1.00}
\definecolor{CadetBlue2}{rgb}{0.56,0.90,0.93}
\definecolor{CadetBlue3}{rgb}{0.48,0.77,0.80}
\definecolor{CadetBlue4}{rgb}{0.33,0.53,0.55}
\definecolor{CadetBlue}{rgb}{0.37,0.62,0.63}
\definecolor{CornflowerBlue}{rgb}{0.39,0.58,0.93}
\definecolor{DarkBlue}{rgb}{0.00,0.00,0.55}
\definecolor{DarkCyan}{rgb}{0.00,0.55,0.55}
\definecolor{DarkGoldenrod1}{rgb}{1.00,0.73,0.06}
\definecolor{DarkGoldenrod2}{rgb}{0.93,0.68,0.05}
\definecolor{DarkGoldenrod3}{rgb}{0.80,0.58,0.05}
\definecolor{DarkGoldenrod4}{rgb}{0.55,0.40,0.03}
\definecolor{DarkGoldenrod}{rgb}{0.72,0.53,0.04}
\definecolor{DarkGray}{rgb}{0.66,0.66,0.66}
\definecolor{DarkGreen}{rgb}{0.00,0.39,0.00}
\definecolor{DarkGrey}{rgb}{0.66,0.66,0.66}
\definecolor{DarkKhaki}{rgb}{0.74,0.72,0.42}
\definecolor{DarkMagenta}{rgb}{0.55,0.00,0.55}
\definecolor{DarkOliveGreen1}{rgb}{0.79,1.00,0.44}
\definecolor{DarkOliveGreen2}{rgb}{0.74,0.93,0.41}
\definecolor{DarkOliveGreen3}{rgb}{0.64,0.80,0.35}
\definecolor{DarkOliveGreen4}{rgb}{0.43,0.55,0.24}
\definecolor{DarkOliveGreen}{rgb}{0.33,0.42,0.18}
\definecolor{DarkOrange1}{rgb}{1.00,0.50,0.00}
\definecolor{DarkOrange2}{rgb}{0.93,0.46,0.00}
\definecolor{DarkOrange3}{rgb}{0.80,0.40,0.00}
\definecolor{DarkOrange4}{rgb}{0.55,0.27,0.00}
\definecolor{DarkOrange}{rgb}{1.00,0.55,0.00}
\definecolor{DarkOrchid1}{rgb}{0.75,0.24,1.00}
\definecolor{DarkOrchid2}{rgb}{0.70,0.23,0.93}
\definecolor{DarkOrchid3}{rgb}{0.60,0.20,0.80}
\definecolor{DarkOrchid4}{rgb}{0.41,0.13,0.55}
\definecolor{DarkOrchid}{rgb}{0.60,0.20,0.80}
\definecolor{DarkRed}{rgb}{0.55,0.00,0.00}
\definecolor{DarkSalmon}{rgb}{0.91,0.59,0.48}
\definecolor{DarkSeaGreen1}{rgb}{0.76,1.00,0.76}
\definecolor{DarkSeaGreen2}{rgb}{0.71,0.93,0.71}
\definecolor{DarkSeaGreen3}{rgb}{0.61,0.80,0.61}
\definecolor{DarkSeaGreen4}{rgb}{0.41,0.55,0.41}
\definecolor{DarkSeaGreen}{rgb}{0.56,0.74,0.56}
\definecolor{DarkSlateBlue}{rgb}{0.28,0.24,0.55}
\definecolor{DarkSlateGray1}{rgb}{0.59,1.00,1.00}
\definecolor{DarkSlateGray2}{rgb}{0.55,0.93,0.93}
\definecolor{DarkSlateGray3}{rgb}{0.47,0.80,0.80}
\definecolor{DarkSlateGray4}{rgb}{0.32,0.55,0.55}
\definecolor{DarkSlateGray}{rgb}{0.18,0.31,0.31}
\definecolor{DarkSlateGrey}{rgb}{0.18,0.31,0.31}
\definecolor{DarkTurquoise}{rgb}{0.00,0.81,0.82}
\definecolor{DarkViolet}{rgb}{0.58,0.00,0.83}
\definecolor{DeepPink1}{rgb}{1.00,0.08,0.58}
\definecolor{DeepPink2}{rgb}{0.93,0.07,0.54}
\definecolor{DeepPink3}{rgb}{0.80,0.06,0.46}
\definecolor{DeepPink4}{rgb}{0.55,0.04,0.31}
\definecolor{DeepPink}{rgb}{1.00,0.08,0.58}
\definecolor{DeepSkyBlue1}{rgb}{0.00,0.75,1.00}
\definecolor{DeepSkyBlue2}{rgb}{0.00,0.70,0.93}
\definecolor{DeepSkyBlue3}{rgb}{0.00,0.60,0.80}
\definecolor{DeepSkyBlue4}{rgb}{0.00,0.41,0.55}
\definecolor{DeepSkyBlue}{rgb}{0.00,0.75,1.00}
\definecolor{DimGray}{rgb}{0.41,0.41,0.41}
\definecolor{DimGrey}{rgb}{0.41,0.41,0.41}
\definecolor{DodgerBlue1}{rgb}{0.12,0.56,1.00}
\definecolor{DodgerBlue2}{rgb}{0.11,0.53,0.93}
\definecolor{DodgerBlue3}{rgb}{0.09,0.45,0.80}
\definecolor{DodgerBlue4}{rgb}{0.06,0.31,0.55}
\definecolor{DodgerBlue}{rgb}{0.12,0.56,1.00}
\definecolor{FloralWhite}{rgb}{1.00,0.98,0.94}
\definecolor{ForestGreen}{rgb}{0.13,0.55,0.13}
\definecolor{GhostWhite}{rgb}{0.97,0.97,1.00}
\definecolor{GreenYellow}{rgb}{0.68,1.00,0.18}
\definecolor{HotPink1}{rgb}{1.00,0.43,0.71}
\definecolor{HotPink2}{rgb}{0.93,0.42,0.65}
\definecolor{HotPink3}{rgb}{0.80,0.38,0.56}
\definecolor{HotPink4}{rgb}{0.55,0.23,0.38}
\definecolor{HotPink}{rgb}{1.00,0.41,0.71}
\definecolor{IndianRed1}{rgb}{1.00,0.42,0.42}
\definecolor{IndianRed2}{rgb}{0.93,0.39,0.39}
\definecolor{IndianRed3}{rgb}{0.80,0.33,0.33}
\definecolor{IndianRed4}{rgb}{0.55,0.23,0.23}
\definecolor{IndianRed}{rgb}{0.80,0.36,0.36}
\definecolor{LavenderBlush1}{rgb}{1.00,0.94,0.96}
\definecolor{LavenderBlush2}{rgb}{0.93,0.88,0.90}
\definecolor{LavenderBlush3}{rgb}{0.80,0.76,0.77}
\definecolor{LavenderBlush4}{rgb}{0.55,0.51,0.53}
\definecolor{LavenderBlush}{rgb}{1.00,0.94,0.96}
\definecolor{LawnGreen}{rgb}{0.49,0.99,0.00}
\definecolor{LemonChiffon1}{rgb}{1.00,0.98,0.80}
\definecolor{LemonChiffon2}{rgb}{0.93,0.91,0.75}
\definecolor{LemonChiffon3}{rgb}{0.80,0.79,0.65}
\definecolor{LemonChiffon4}{rgb}{0.55,0.54,0.44}
\definecolor{LemonChiffon}{rgb}{1.00,0.98,0.80}
\definecolor{LightBlue1}{rgb}{0.75,0.94,1.00}
\definecolor{LightBlue2}{rgb}{0.70,0.87,0.93}
\definecolor{LightBlue3}{rgb}{0.60,0.75,0.80}
\definecolor{LightBlue4}{rgb}{0.41,0.51,0.55}
\definecolor{LightBlue}{rgb}{0.68,0.85,0.90}
\definecolor{LightCoral}{rgb}{0.94,0.50,0.50}
\definecolor{LightCyan1}{rgb}{0.88,1.00,1.00}
\definecolor{LightCyan2}{rgb}{0.82,0.93,0.93}
\definecolor{LightCyan3}{rgb}{0.71,0.80,0.80}
\definecolor{LightCyan4}{rgb}{0.48,0.55,0.55}
\definecolor{LightCyan}{rgb}{0.88,1.00,1.00}
\definecolor{LightGoldenrod1}{rgb}{1.00,0.93,0.55}
\definecolor{LightGoldenrod2}{rgb}{0.93,0.86,0.51}
\definecolor{LightGoldenrod3}{rgb}{0.80,0.75,0.44}
\definecolor{LightGoldenrod4}{rgb}{0.55,0.51,0.30}
\definecolor{LightGoldenrodYellow}{rgb}{0.98,0.98,0.82}
\definecolor{LightGoldenrod}{rgb}{0.93,0.87,0.51}
\definecolor{LightGray}{rgb}{0.83,0.83,0.83}
\definecolor{LightGreen}{rgb}{0.56,0.93,0.56}
\definecolor{LightGrey}{rgb}{0.83,0.83,0.83}
\definecolor{LightPink1}{rgb}{1.00,0.68,0.73}
\definecolor{LightPink2}{rgb}{0.93,0.64,0.68}
\definecolor{LightPink3}{rgb}{0.80,0.55,0.58}
\definecolor{LightPink4}{rgb}{0.55,0.37,0.40}
\definecolor{LightPink}{rgb}{1.00,0.71,0.76}
\definecolor{LightSalmon1}{rgb}{1.00,0.63,0.48}
\definecolor{LightSalmon2}{rgb}{0.93,0.58,0.45}
\definecolor{LightSalmon3}{rgb}{0.80,0.51,0.38}
\definecolor{LightSalmon4}{rgb}{0.55,0.34,0.26}
\definecolor{LightSalmon}{rgb}{1.00,0.63,0.48}
\definecolor{LightSeaGreen}{rgb}{0.13,0.70,0.67}
\definecolor{LightSkyBlue1}{rgb}{0.69,0.89,1.00}
\definecolor{LightSkyBlue2}{rgb}{0.64,0.83,0.93}
\definecolor{LightSkyBlue3}{rgb}{0.55,0.71,0.80}
\definecolor{LightSkyBlue4}{rgb}{0.38,0.48,0.55}
\definecolor{LightSkyBlue}{rgb}{0.53,0.81,0.98}
\definecolor{LightSlateBlue}{rgb}{0.52,0.44,1.00}
\definecolor{LightSlateGray}{rgb}{0.47,0.53,0.60}
\definecolor{LightSlateGrey}{rgb}{0.47,0.53,0.60}
\definecolor{LightSteelBlue1}{rgb}{0.79,0.88,1.00}
\definecolor{LightSteelBlue2}{rgb}{0.74,0.82,0.93}
\definecolor{LightSteelBlue3}{rgb}{0.64,0.71,0.80}
\definecolor{LightSteelBlue4}{rgb}{0.43,0.48,0.55}
\definecolor{LightSteelBlue}{rgb}{0.69,0.77,0.87}
\definecolor{LightYellow1}{rgb}{1.00,1.00,0.88}
\definecolor{LightYellow2}{rgb}{0.93,0.93,0.82}
\definecolor{LightYellow3}{rgb}{0.80,0.80,0.71}
\definecolor{LightYellow4}{rgb}{0.55,0.55,0.48}
\definecolor{LightYellow}{rgb}{1.00,1.00,0.88}
\definecolor{LimeGreen}{rgb}{0.20,0.80,0.20}
\definecolor{MediumAquamarine}{rgb}{0.40,0.80,0.67}
\definecolor{MediumBlue}{rgb}{0.00,0.00,0.80}
\definecolor{MediumOrchid1}{rgb}{0.88,0.40,1.00}
\definecolor{MediumOrchid2}{rgb}{0.82,0.37,0.93}
\definecolor{MediumOrchid3}{rgb}{0.71,0.32,0.80}
\definecolor{MediumOrchid4}{rgb}{0.48,0.22,0.55}
\definecolor{MediumOrchid}{rgb}{0.73,0.33,0.83}
\definecolor{MediumPurple1}{rgb}{0.67,0.51,1.00}
\definecolor{MediumPurple2}{rgb}{0.62,0.47,0.93}
\definecolor{MediumPurple3}{rgb}{0.54,0.41,0.80}
\definecolor{MediumPurple4}{rgb}{0.36,0.28,0.55}
\definecolor{MediumPurple}{rgb}{0.58,0.44,0.86}
\definecolor{MediumSeaGreen}{rgb}{0.24,0.70,0.44}
\definecolor{MediumSlateBlue}{rgb}{0.48,0.41,0.93}
\definecolor{MediumSpringGreen}{rgb}{0.00,0.98,0.60}
\definecolor{MediumTurquoise}{rgb}{0.28,0.82,0.80}
\definecolor{MediumVioletRed}{rgb}{0.78,0.08,0.52}
\definecolor{MidnightBlue}{rgb}{0.10,0.10,0.44}
\definecolor{MintCream}{rgb}{0.96,1.00,0.98}
\definecolor{MistyRose1}{rgb}{1.00,0.89,0.88}
\definecolor{MistyRose2}{rgb}{0.93,0.84,0.82}
\definecolor{MistyRose3}{rgb}{0.80,0.72,0.71}
\definecolor{MistyRose4}{rgb}{0.55,0.49,0.48}
\definecolor{MistyRose}{rgb}{1.00,0.89,0.88}
\definecolor{NavajoWhite1}{rgb}{1.00,0.87,0.68}
\definecolor{NavajoWhite2}{rgb}{0.93,0.81,0.63}
\definecolor{NavajoWhite3}{rgb}{0.80,0.70,0.55}
\definecolor{NavajoWhite4}{rgb}{0.55,0.47,0.37}
\definecolor{NavajoWhite}{rgb}{1.00,0.87,0.68}
\definecolor{NavyBlue}{rgb}{0.00,0.00,0.50}
\definecolor{OldLace}{rgb}{0.99,0.96,0.90}
\definecolor{OliveDrab1}{rgb}{0.75,1.00,0.24}
\definecolor{OliveDrab2}{rgb}{0.70,0.93,0.23}
\definecolor{OliveDrab3}{rgb}{0.60,0.80,0.20}
\definecolor{OliveDrab4}{rgb}{0.41,0.55,0.13}
\definecolor{OliveDrab}{rgb}{0.42,0.56,0.14}
\definecolor{OrangeRed1}{rgb}{1.00,0.27,0.00}
\definecolor{OrangeRed2}{rgb}{0.93,0.25,0.00}
\definecolor{OrangeRed3}{rgb}{0.80,0.22,0.00}
\definecolor{OrangeRed4}{rgb}{0.55,0.15,0.00}
\definecolor{OrangeRed}{rgb}{1.00,0.27,0.00}
\definecolor{PaleGoldenrod}{rgb}{0.93,0.91,0.67}
\definecolor{PaleGreen1}{rgb}{0.60,1.00,0.60}
\definecolor{PaleGreen2}{rgb}{0.56,0.93,0.56}
\definecolor{PaleGreen3}{rgb}{0.49,0.80,0.49}
\definecolor{PaleGreen4}{rgb}{0.33,0.55,0.33}
\definecolor{PaleGreen}{rgb}{0.60,0.98,0.60}
\definecolor{PaleTurquoise1}{rgb}{0.73,1.00,1.00}
\definecolor{PaleTurquoise2}{rgb}{0.68,0.93,0.93}
\definecolor{PaleTurquoise3}{rgb}{0.59,0.80,0.80}
\definecolor{PaleTurquoise4}{rgb}{0.40,0.55,0.55}
\definecolor{PaleTurquoise}{rgb}{0.69,0.93,0.93}
\definecolor{PaleVioletRed1}{rgb}{1.00,0.51,0.67}
\definecolor{PaleVioletRed2}{rgb}{0.93,0.47,0.62}
\definecolor{PaleVioletRed3}{rgb}{0.80,0.41,0.54}
\definecolor{PaleVioletRed4}{rgb}{0.55,0.28,0.36}
\definecolor{PaleVioletRed}{rgb}{0.86,0.44,0.58}
\definecolor{PapayaWhip}{rgb}{1.00,0.94,0.84}
\definecolor{PeachPuff1}{rgb}{1.00,0.85,0.73}
\definecolor{PeachPuff2}{rgb}{0.93,0.80,0.68}
\definecolor{PeachPuff3}{rgb}{0.80,0.69,0.58}
\definecolor{PeachPuff4}{rgb}{0.55,0.47,0.40}
\definecolor{PeachPuff}{rgb}{1.00,0.85,0.73}
\definecolor{PowderBlue}{rgb}{0.69,0.88,0.90}
\definecolor{RosyBrown1}{rgb}{1.00,0.76,0.76}
\definecolor{RosyBrown2}{rgb}{0.93,0.71,0.71}
\definecolor{RosyBrown3}{rgb}{0.80,0.61,0.61}
\definecolor{RosyBrown4}{rgb}{0.55,0.41,0.41}
\definecolor{RosyBrown}{rgb}{0.74,0.56,0.56}
\definecolor{RoyalBlue1}{rgb}{0.28,0.46,1.00}
\definecolor{RoyalBlue2}{rgb}{0.26,0.43,0.93}
\definecolor{RoyalBlue3}{rgb}{0.23,0.37,0.80}
\definecolor{RoyalBlue4}{rgb}{0.15,0.25,0.55}
\definecolor{RoyalBlue}{rgb}{0.25,0.41,0.88}
\definecolor{SaddleBrown}{rgb}{0.55,0.27,0.07}
\definecolor{SandyBrown}{rgb}{0.96,0.64,0.38}
\definecolor{SeaGreen1}{rgb}{0.33,1.00,0.62}
\definecolor{SeaGreen2}{rgb}{0.31,0.93,0.58}
\definecolor{SeaGreen3}{rgb}{0.26,0.80,0.50}
\definecolor{SeaGreen4}{rgb}{0.18,0.55,0.34}
\definecolor{SeaGreen}{rgb}{0.18,0.55,0.34}
\definecolor{SkyBlue1}{rgb}{0.53,0.81,1.00}
\definecolor{SkyBlue2}{rgb}{0.49,0.75,0.93}
\definecolor{SkyBlue3}{rgb}{0.42,0.65,0.80}
\definecolor{SkyBlue4}{rgb}{0.29,0.44,0.55}
\definecolor{SkyBlue}{rgb}{0.53,0.81,0.92}
\definecolor{SlateBlue1}{rgb}{0.51,0.44,1.00}
\definecolor{SlateBlue2}{rgb}{0.48,0.40,0.93}
\definecolor{SlateBlue3}{rgb}{0.41,0.35,0.80}
\definecolor{SlateBlue4}{rgb}{0.28,0.24,0.55}
\definecolor{SlateBlue}{rgb}{0.42,0.35,0.80}
\definecolor{SlateGray1}{rgb}{0.78,0.89,1.00}
\definecolor{SlateGray2}{rgb}{0.73,0.83,0.93}
\definecolor{SlateGray3}{rgb}{0.62,0.71,0.80}
\definecolor{SlateGray4}{rgb}{0.42,0.48,0.55}
\definecolor{SlateGray}{rgb}{0.44,0.50,0.56}
\definecolor{SlateGrey}{rgb}{0.44,0.50,0.56}
\definecolor{SpringGreen1}{rgb}{0.00,1.00,0.50}
\definecolor{SpringGreen2}{rgb}{0.00,0.93,0.46}
\definecolor{SpringGreen3}{rgb}{0.00,0.80,0.40}
\definecolor{SpringGreen4}{rgb}{0.00,0.55,0.27}
\definecolor{SpringGreen}{rgb}{0.00,1.00,0.50}
\definecolor{SteelBlue1}{rgb}{0.39,0.72,1.00}
\definecolor{SteelBlue2}{rgb}{0.36,0.67,0.93}
\definecolor{SteelBlue3}{rgb}{0.31,0.58,0.80}
\definecolor{SteelBlue4}{rgb}{0.21,0.39,0.55}
\definecolor{SteelBlue}{rgb}{0.27,0.51,0.71}
\definecolor{VioletRed1}{rgb}{1.00,0.24,0.59}
\definecolor{VioletRed2}{rgb}{0.93,0.23,0.55}
\definecolor{VioletRed3}{rgb}{0.80,0.20,0.47}
\definecolor{VioletRed4}{rgb}{0.55,0.13,0.32}
\definecolor{VioletRed}{rgb}{0.82,0.13,0.56}
\definecolor{WhiteSmoke}{rgb}{0.96,0.96,0.96}
\definecolor{YellowGreen}{rgb}{0.60,0.80,0.20}
\definecolor{aliceblue}{rgb}{0.94,0.97,1.00}
\definecolor{antiquewhite}{rgb}{0.98,0.92,0.84}
\definecolor{aquamarine1}{rgb}{0.50,1.00,0.83}
\definecolor{aquamarine2}{rgb}{0.46,0.93,0.78}
\definecolor{aquamarine3}{rgb}{0.40,0.80,0.67}
\definecolor{aquamarine4}{rgb}{0.27,0.55,0.45}
\definecolor{aquamarine}{rgb}{0.50,1.00,0.83}
\definecolor{azure1}{rgb}{0.94,1.00,1.00}
\definecolor{azure2}{rgb}{0.88,0.93,0.93}
\definecolor{azure3}{rgb}{0.76,0.80,0.80}
\definecolor{azure4}{rgb}{0.51,0.55,0.55}
\definecolor{azure}{rgb}{0.94,1.00,1.00}
\definecolor{beige}{rgb}{0.96,0.96,0.86}
\definecolor{bisque1}{rgb}{1.00,0.89,0.77}
\definecolor{bisque2}{rgb}{0.93,0.84,0.72}
\definecolor{bisque3}{rgb}{0.80,0.72,0.62}
\definecolor{bisque4}{rgb}{0.55,0.49,0.42}
\definecolor{bisque}{rgb}{1.00,0.89,0.77}
\definecolor{black}{rgb}{0.00,0.00,0.00}
\definecolor{blanchedalmond}{rgb}{1.00,0.92,0.80}
\definecolor{blue1}{rgb}{0.00,0.00,1.00}
\definecolor{blue2}{rgb}{0.00,0.00,0.93}
\definecolor{blue3}{rgb}{0.00,0.00,0.80}
\definecolor{blue4}{rgb}{0.00,0.00,0.55}
\definecolor{blueviolet}{rgb}{0.54,0.17,0.89}
\definecolor{blue}{rgb}{0.00,0.00,1.00}
\definecolor{brown1}{rgb}{1.00,0.25,0.25}
\definecolor{brown2}{rgb}{0.93,0.23,0.23}
\definecolor{brown3}{rgb}{0.80,0.20,0.20}
\definecolor{brown4}{rgb}{0.55,0.14,0.14}
\definecolor{brown}{rgb}{0.65,0.16,0.16}
\definecolor{burlywood1}{rgb}{1.00,0.83,0.61}
\definecolor{burlywood2}{rgb}{0.93,0.77,0.57}
\definecolor{burlywood3}{rgb}{0.80,0.67,0.49}
\definecolor{burlywood4}{rgb}{0.55,0.45,0.33}
\definecolor{burlywood}{rgb}{0.87,0.72,0.53}
\definecolor{cadetblue}{rgb}{0.37,0.62,0.63}
\definecolor{chartreuse1}{rgb}{0.50,1.00,0.00}
\definecolor{chartreuse2}{rgb}{0.46,0.93,0.00}
\definecolor{chartreuse3}{rgb}{0.40,0.80,0.00}
\definecolor{chartreuse4}{rgb}{0.27,0.55,0.00}
\definecolor{chartreuse}{rgb}{0.50,1.00,0.00}
\definecolor{chocolate1}{rgb}{1.00,0.50,0.14}
\definecolor{chocolate2}{rgb}{0.93,0.46,0.13}
\definecolor{chocolate3}{rgb}{0.80,0.40,0.11}
\definecolor{chocolate4}{rgb}{0.55,0.27,0.07}
\definecolor{chocolate}{rgb}{0.82,0.41,0.12}
\definecolor{coral1}{rgb}{1.00,0.45,0.34}
\definecolor{coral2}{rgb}{0.93,0.42,0.31}
\definecolor{coral3}{rgb}{0.80,0.36,0.27}
\definecolor{coral4}{rgb}{0.55,0.24,0.18}
\definecolor{coral}{rgb}{1.00,0.50,0.31}
\definecolor{cornflowerblue}{rgb}{0.39,0.58,0.93}
\definecolor{cornsilk1}{rgb}{1.00,0.97,0.86}
\definecolor{cornsilk2}{rgb}{0.93,0.91,0.80}
\definecolor{cornsilk3}{rgb}{0.80,0.78,0.69}
\definecolor{cornsilk4}{rgb}{0.55,0.53,0.47}
\definecolor{cornsilk}{rgb}{1.00,0.97,0.86}
\definecolor{cyan1}{rgb}{0.00,1.00,1.00}
\definecolor{cyan2}{rgb}{0.00,0.93,0.93}
\definecolor{cyan3}{rgb}{0.00,0.80,0.80}
\definecolor{cyan4}{rgb}{0.00,0.55,0.55}
\definecolor{cyan}{rgb}{0.00,1.00,1.00}
\definecolor{darkblue}{rgb}{0.00,0.00,0.55}
\definecolor{darkcyan}{rgb}{0.00,0.55,0.55}
\definecolor{darkgoldenrod}{rgb}{0.72,0.53,0.04}
\definecolor{darkgray}{rgb}{0.66,0.66,0.66}
\definecolor{darkgreen}{rgb}{0.00,0.39,0.00}
\definecolor{darkgrey}{rgb}{0.66,0.66,0.66}
\definecolor{darkkhaki}{rgb}{0.74,0.72,0.42}
\definecolor{darkmagenta}{rgb}{0.55,0.00,0.55}
\definecolor{darkolive}{rgb}{0.33,0.42,0.18}
\definecolor{darkorange}{rgb}{1.00,0.55,0.00}
\definecolor{darkorchid}{rgb}{0.60,0.20,0.80}
\definecolor{darkred}{rgb}{0.55,0.00,0.00}
\definecolor{darksalmon}{rgb}{0.91,0.59,0.48}
\definecolor{darksea}{rgb}{0.56,0.74,0.56}
\definecolor{darkslate}{rgb}{0.18,0.31,0.31}
\definecolor{darkslate}{rgb}{0.18,0.31,0.31}
\definecolor{darkslate}{rgb}{0.28,0.24,0.55}
\definecolor{darkturquoise}{rgb}{0.00,0.81,0.82}
\definecolor{darkviolet}{rgb}{0.58,0.00,0.83}
\definecolor{deeppink}{rgb}{1.00,0.08,0.58}
\definecolor{deepsky}{rgb}{0.00,0.75,1.00}
\definecolor{dimgray}{rgb}{0.41,0.41,0.41}
\definecolor{dimgrey}{rgb}{0.41,0.41,0.41}
\definecolor{dodgerblue}{rgb}{0.12,0.56,1.00}
\definecolor{firebrick1}{rgb}{1.00,0.19,0.19}
\definecolor{firebrick2}{rgb}{0.93,0.17,0.17}
\definecolor{firebrick3}{rgb}{0.80,0.15,0.15}
\definecolor{firebrick4}{rgb}{0.55,0.10,0.10}
\definecolor{firebrick}{rgb}{0.70,0.13,0.13}
\definecolor{floralwhite}{rgb}{1.00,0.98,0.94}
\definecolor{forestgreen}{rgb}{0.13,0.55,0.13}
\definecolor{gainsboro}{rgb}{0.86,0.86,0.86}
\definecolor{ghostwhite}{rgb}{0.97,0.97,1.00}
\definecolor{gold1}{rgb}{1.00,0.84,0.00}
\definecolor{gold2}{rgb}{0.93,0.79,0.00}
\definecolor{gold3}{rgb}{0.80,0.68,0.00}
\definecolor{gold4}{rgb}{0.55,0.46,0.00}
\definecolor{goldenrod1}{rgb}{1.00,0.76,0.15}
\definecolor{goldenrod2}{rgb}{0.93,0.71,0.13}
\definecolor{goldenrod3}{rgb}{0.80,0.61,0.11}
\definecolor{goldenrod4}{rgb}{0.55,0.41,0.08}
\definecolor{goldenrod}{rgb}{0.85,0.65,0.13}
\definecolor{gold}{rgb}{1.00,0.84,0.00}
\definecolor{gray0}{rgb}{0.00,0.00,0.00}
\definecolor{gray100}{rgb}{1.00,1.00,1.00}
\definecolor{gray10}{rgb}{0.10,0.10,0.10}
\definecolor{gray11}{rgb}{0.11,0.11,0.11}
\definecolor{gray12}{rgb}{0.12,0.12,0.12}
\definecolor{gray13}{rgb}{0.13,0.13,0.13}
\definecolor{gray14}{rgb}{0.14,0.14,0.14}
\definecolor{gray15}{rgb}{0.15,0.15,0.15}
\definecolor{gray16}{rgb}{0.16,0.16,0.16}
\definecolor{gray17}{rgb}{0.17,0.17,0.17}
\definecolor{gray18}{rgb}{0.18,0.18,0.18}
\definecolor{gray19}{rgb}{0.19,0.19,0.19}
\definecolor{gray1}{rgb}{0.01,0.01,0.01}
\definecolor{gray20}{rgb}{0.20,0.20,0.20}
\definecolor{gray21}{rgb}{0.21,0.21,0.21}
\definecolor{gray22}{rgb}{0.22,0.22,0.22}
\definecolor{gray23}{rgb}{0.23,0.23,0.23}
\definecolor{gray24}{rgb}{0.24,0.24,0.24}
\definecolor{gray25}{rgb}{0.25,0.25,0.25}
\definecolor{gray26}{rgb}{0.26,0.26,0.26}
\definecolor{gray27}{rgb}{0.27,0.27,0.27}
\definecolor{gray28}{rgb}{0.28,0.28,0.28}
\definecolor{gray29}{rgb}{0.29,0.29,0.29}
\definecolor{gray2}{rgb}{0.02,0.02,0.02}
\definecolor{gray30}{rgb}{0.30,0.30,0.30}
\definecolor{gray31}{rgb}{0.31,0.31,0.31}
\definecolor{gray32}{rgb}{0.32,0.32,0.32}
\definecolor{gray33}{rgb}{0.33,0.33,0.33}
\definecolor{gray34}{rgb}{0.34,0.34,0.34}
\definecolor{gray35}{rgb}{0.35,0.35,0.35}
\definecolor{gray36}{rgb}{0.36,0.36,0.36}
\definecolor{gray37}{rgb}{0.37,0.37,0.37}
\definecolor{gray38}{rgb}{0.38,0.38,0.38}
\definecolor{gray39}{rgb}{0.39,0.39,0.39}
\definecolor{gray3}{rgb}{0.03,0.03,0.03}
\definecolor{gray40}{rgb}{0.40,0.40,0.40}
\definecolor{gray41}{rgb}{0.41,0.41,0.41}
\definecolor{gray42}{rgb}{0.42,0.42,0.42}
\definecolor{gray43}{rgb}{0.43,0.43,0.43}
\definecolor{gray44}{rgb}{0.44,0.44,0.44}
\definecolor{gray45}{rgb}{0.45,0.45,0.45}
\definecolor{gray46}{rgb}{0.46,0.46,0.46}
\definecolor{gray47}{rgb}{0.47,0.47,0.47}
\definecolor{gray48}{rgb}{0.48,0.48,0.48}
\definecolor{gray49}{rgb}{0.49,0.49,0.49}
\definecolor{gray4}{rgb}{0.04,0.04,0.04}
\definecolor{gray50}{rgb}{0.50,0.50,0.50}
\definecolor{gray51}{rgb}{0.51,0.51,0.51}
\definecolor{gray52}{rgb}{0.52,0.52,0.52}
\definecolor{gray53}{rgb}{0.53,0.53,0.53}
\definecolor{gray54}{rgb}{0.54,0.54,0.54}
\definecolor{gray55}{rgb}{0.55,0.55,0.55}
\definecolor{gray56}{rgb}{0.56,0.56,0.56}
\definecolor{gray57}{rgb}{0.57,0.57,0.57}
\definecolor{gray58}{rgb}{0.58,0.58,0.58}
\definecolor{gray59}{rgb}{0.59,0.59,0.59}
\definecolor{gray5}{rgb}{0.05,0.05,0.05}
\definecolor{gray60}{rgb}{0.60,0.60,0.60}
\definecolor{gray61}{rgb}{0.61,0.61,0.61}
\definecolor{gray62}{rgb}{0.62,0.62,0.62}
\definecolor{gray63}{rgb}{0.63,0.63,0.63}
\definecolor{gray64}{rgb}{0.64,0.64,0.64}
\definecolor{gray65}{rgb}{0.65,0.65,0.65}
\definecolor{gray66}{rgb}{0.66,0.66,0.66}
\definecolor{gray67}{rgb}{0.67,0.67,0.67}
\definecolor{gray68}{rgb}{0.68,0.68,0.68}
\definecolor{gray69}{rgb}{0.69,0.69,0.69}
\definecolor{gray6}{rgb}{0.06,0.06,0.06}
\definecolor{gray70}{rgb}{0.70,0.70,0.70}
\definecolor{gray71}{rgb}{0.71,0.71,0.71}
\definecolor{gray72}{rgb}{0.72,0.72,0.72}
\definecolor{gray73}{rgb}{0.73,0.73,0.73}
\definecolor{gray74}{rgb}{0.74,0.74,0.74}
\definecolor{gray75}{rgb}{0.75,0.75,0.75}
\definecolor{gray76}{rgb}{0.76,0.76,0.76}
\definecolor{gray77}{rgb}{0.77,0.77,0.77}
\definecolor{gray78}{rgb}{0.78,0.78,0.78}
\definecolor{gray79}{rgb}{0.79,0.79,0.79}
\definecolor{gray7}{rgb}{0.07,0.07,0.07}
\definecolor{gray80}{rgb}{0.80,0.80,0.80}
\definecolor{gray81}{rgb}{0.81,0.81,0.81}
\definecolor{gray82}{rgb}{0.82,0.82,0.82}
\definecolor{gray83}{rgb}{0.83,0.83,0.83}
\definecolor{gray84}{rgb}{0.84,0.84,0.84}
\definecolor{gray85}{rgb}{0.85,0.85,0.85}
\definecolor{gray86}{rgb}{0.86,0.86,0.86}
\definecolor{gray87}{rgb}{0.87,0.87,0.87}
\definecolor{gray88}{rgb}{0.88,0.88,0.88}
\definecolor{gray89}{rgb}{0.89,0.89,0.89}
\definecolor{gray8}{rgb}{0.08,0.08,0.08}
\definecolor{gray90}{rgb}{0.90,0.90,0.90}
\definecolor{gray91}{rgb}{0.91,0.91,0.91}
\definecolor{gray92}{rgb}{0.92,0.92,0.92}
\definecolor{gray93}{rgb}{0.93,0.93,0.93}
\definecolor{gray94}{rgb}{0.94,0.94,0.94}
\definecolor{gray95}{rgb}{0.95,0.95,0.95}
\definecolor{gray96}{rgb}{0.96,0.96,0.96}
\definecolor{gray97}{rgb}{0.97,0.97,0.97}
\definecolor{gray98}{rgb}{0.98,0.98,0.98}
\definecolor{gray99}{rgb}{0.99,0.99,0.99}
\definecolor{gray9}{rgb}{0.09,0.09,0.09}
\definecolor{gray}{rgb}{0.75,0.75,0.75}
\definecolor{green1}{rgb}{0.00,1.00,0.00}
\definecolor{green2}{rgb}{0.00,0.93,0.00}
\definecolor{green3}{rgb}{0.00,0.80,0.00}
\definecolor{green4}{rgb}{0.00,0.55,0.00}
\definecolor{greenyellow}{rgb}{0.68,1.00,0.18}
\definecolor{green}{rgb}{0.00,1.00,0.00}
\definecolor{grey0}{rgb}{0.00,0.00,0.00}
\definecolor{grey100}{rgb}{1.00,1.00,1.00}
\definecolor{grey10}{rgb}{0.10,0.10,0.10}
\definecolor{grey11}{rgb}{0.11,0.11,0.11}
\definecolor{grey12}{rgb}{0.12,0.12,0.12}
\definecolor{grey13}{rgb}{0.13,0.13,0.13}
\definecolor{grey14}{rgb}{0.14,0.14,0.14}
\definecolor{grey15}{rgb}{0.15,0.15,0.15}
\definecolor{grey16}{rgb}{0.16,0.16,0.16}
\definecolor{grey17}{rgb}{0.17,0.17,0.17}
\definecolor{grey18}{rgb}{0.18,0.18,0.18}
\definecolor{grey19}{rgb}{0.19,0.19,0.19}
\definecolor{grey1}{rgb}{0.01,0.01,0.01}
\definecolor{grey20}{rgb}{0.20,0.20,0.20}
\definecolor{grey21}{rgb}{0.21,0.21,0.21}
\definecolor{grey22}{rgb}{0.22,0.22,0.22}
\definecolor{grey23}{rgb}{0.23,0.23,0.23}
\definecolor{grey24}{rgb}{0.24,0.24,0.24}
\definecolor{grey25}{rgb}{0.25,0.25,0.25}
\definecolor{grey26}{rgb}{0.26,0.26,0.26}
\definecolor{grey27}{rgb}{0.27,0.27,0.27}
\definecolor{grey28}{rgb}{0.28,0.28,0.28}
\definecolor{grey29}{rgb}{0.29,0.29,0.29}
\definecolor{grey2}{rgb}{0.02,0.02,0.02}
\definecolor{grey30}{rgb}{0.30,0.30,0.30}
\definecolor{grey31}{rgb}{0.31,0.31,0.31}
\definecolor{grey32}{rgb}{0.32,0.32,0.32}
\definecolor{grey33}{rgb}{0.33,0.33,0.33}
\definecolor{grey34}{rgb}{0.34,0.34,0.34}
\definecolor{grey35}{rgb}{0.35,0.35,0.35}
\definecolor{grey36}{rgb}{0.36,0.36,0.36}
\definecolor{grey37}{rgb}{0.37,0.37,0.37}
\definecolor{grey38}{rgb}{0.38,0.38,0.38}
\definecolor{grey39}{rgb}{0.39,0.39,0.39}
\definecolor{grey3}{rgb}{0.03,0.03,0.03}
\definecolor{grey40}{rgb}{0.40,0.40,0.40}
\definecolor{grey41}{rgb}{0.41,0.41,0.41}
\definecolor{grey42}{rgb}{0.42,0.42,0.42}
\definecolor{grey43}{rgb}{0.43,0.43,0.43}
\definecolor{grey44}{rgb}{0.44,0.44,0.44}
\definecolor{grey45}{rgb}{0.45,0.45,0.45}
\definecolor{grey46}{rgb}{0.46,0.46,0.46}
\definecolor{grey47}{rgb}{0.47,0.47,0.47}
\definecolor{grey48}{rgb}{0.48,0.48,0.48}
\definecolor{grey49}{rgb}{0.49,0.49,0.49}
\definecolor{grey4}{rgb}{0.04,0.04,0.04}
\definecolor{grey50}{rgb}{0.50,0.50,0.50}
\definecolor{grey51}{rgb}{0.51,0.51,0.51}
\definecolor{grey52}{rgb}{0.52,0.52,0.52}
\definecolor{grey53}{rgb}{0.53,0.53,0.53}
\definecolor{grey54}{rgb}{0.54,0.54,0.54}
\definecolor{grey55}{rgb}{0.55,0.55,0.55}
\definecolor{grey56}{rgb}{0.56,0.56,0.56}
\definecolor{grey57}{rgb}{0.57,0.57,0.57}
\definecolor{grey58}{rgb}{0.58,0.58,0.58}
\definecolor{grey59}{rgb}{0.59,0.59,0.59}
\definecolor{grey5}{rgb}{0.05,0.05,0.05}
\definecolor{grey60}{rgb}{0.60,0.60,0.60}
\definecolor{grey61}{rgb}{0.61,0.61,0.61}
\definecolor{grey62}{rgb}{0.62,0.62,0.62}
\definecolor{grey63}{rgb}{0.63,0.63,0.63}
\definecolor{grey64}{rgb}{0.64,0.64,0.64}
\definecolor{grey65}{rgb}{0.65,0.65,0.65}
\definecolor{grey66}{rgb}{0.66,0.66,0.66}
\definecolor{grey67}{rgb}{0.67,0.67,0.67}
\definecolor{grey68}{rgb}{0.68,0.68,0.68}
\definecolor{grey69}{rgb}{0.69,0.69,0.69}
\definecolor{grey6}{rgb}{0.06,0.06,0.06}
\definecolor{grey70}{rgb}{0.70,0.70,0.70}
\definecolor{grey71}{rgb}{0.71,0.71,0.71}
\definecolor{grey72}{rgb}{0.72,0.72,0.72}
\definecolor{grey73}{rgb}{0.73,0.73,0.73}
\definecolor{grey74}{rgb}{0.74,0.74,0.74}
\definecolor{grey75}{rgb}{0.75,0.75,0.75}
\definecolor{grey76}{rgb}{0.76,0.76,0.76}
\definecolor{grey77}{rgb}{0.77,0.77,0.77}
\definecolor{grey78}{rgb}{0.78,0.78,0.78}
\definecolor{grey79}{rgb}{0.79,0.79,0.79}
\definecolor{grey7}{rgb}{0.07,0.07,0.07}
\definecolor{grey80}{rgb}{0.80,0.80,0.80}
\definecolor{grey81}{rgb}{0.81,0.81,0.81}
\definecolor{grey82}{rgb}{0.82,0.82,0.82}
\definecolor{grey83}{rgb}{0.83,0.83,0.83}
\definecolor{grey84}{rgb}{0.84,0.84,0.84}
\definecolor{grey85}{rgb}{0.85,0.85,0.85}
\definecolor{grey86}{rgb}{0.86,0.86,0.86}
\definecolor{grey87}{rgb}{0.87,0.87,0.87}
\definecolor{grey88}{rgb}{0.88,0.88,0.88}
\definecolor{grey89}{rgb}{0.89,0.89,0.89}
\definecolor{grey8}{rgb}{0.08,0.08,0.08}
\definecolor{grey90}{rgb}{0.90,0.90,0.90}
\definecolor{grey91}{rgb}{0.91,0.91,0.91}
\definecolor{grey92}{rgb}{0.92,0.92,0.92}
\definecolor{grey93}{rgb}{0.93,0.93,0.93}
\definecolor{grey94}{rgb}{0.94,0.94,0.94}
\definecolor{grey95}{rgb}{0.95,0.95,0.95}
\definecolor{grey96}{rgb}{0.96,0.96,0.96}
\definecolor{grey97}{rgb}{0.97,0.97,0.97}
\definecolor{grey98}{rgb}{0.98,0.98,0.98}
\definecolor{grey99}{rgb}{0.99,0.99,0.99}
\definecolor{grey9}{rgb}{0.09,0.09,0.09}
\definecolor{grey}{rgb}{0.75,0.75,0.75}
\definecolor{honeydew1}{rgb}{0.94,1.00,0.94}
\definecolor{honeydew2}{rgb}{0.88,0.93,0.88}
\definecolor{honeydew3}{rgb}{0.76,0.80,0.76}
\definecolor{honeydew4}{rgb}{0.51,0.55,0.51}
\definecolor{honeydew}{rgb}{0.94,1.00,0.94}
\definecolor{hotpink}{rgb}{1.00,0.41,0.71}
\definecolor{indianred}{rgb}{0.80,0.36,0.36}
\definecolor{ivory1}{rgb}{1.00,1.00,0.94}
\definecolor{ivory2}{rgb}{0.93,0.93,0.88}
\definecolor{ivory3}{rgb}{0.80,0.80,0.76}
\definecolor{ivory4}{rgb}{0.55,0.55,0.51}
\definecolor{ivory}{rgb}{1.00,1.00,0.94}
\definecolor{khaki1}{rgb}{1.00,0.96,0.56}
\definecolor{khaki2}{rgb}{0.93,0.90,0.52}
\definecolor{khaki3}{rgb}{0.80,0.78,0.45}
\definecolor{khaki4}{rgb}{0.55,0.53,0.31}
\definecolor{khaki}{rgb}{0.94,0.90,0.55}
\definecolor{lavenderblush}{rgb}{1.00,0.94,0.96}
\definecolor{lavender}{rgb}{0.90,0.90,0.98}
\definecolor{lawngreen}{rgb}{0.49,0.99,0.00}
\definecolor{lemonchiffon}{rgb}{1.00,0.98,0.80}
\definecolor{lightblue}{rgb}{0.68,0.85,0.90}
\definecolor{lightcoral}{rgb}{0.94,0.50,0.50}
\definecolor{lightcyan}{rgb}{0.88,1.00,1.00}
\definecolor{lightgoldenrod}{rgb}{0.93,0.87,0.51}
\definecolor{lightgoldenrod}{rgb}{0.98,0.98,0.82}
\definecolor{lightgray}{rgb}{0.83,0.83,0.83}
\definecolor{lightgreen}{rgb}{0.56,0.93,0.56}
\definecolor{lightgrey}{rgb}{0.83,0.83,0.83}
\definecolor{lightpink}{rgb}{1.00,0.71,0.76}
\definecolor{lightsalmon}{rgb}{1.00,0.63,0.48}
\definecolor{lightsea}{rgb}{0.13,0.70,0.67}
\definecolor{lightsky}{rgb}{0.53,0.81,0.98}
\definecolor{lightslate}{rgb}{0.47,0.53,0.60}
\definecolor{lightslate}{rgb}{0.47,0.53,0.60}
\definecolor{lightslate}{rgb}{0.52,0.44,1.00}
\definecolor{lightsteel}{rgb}{0.69,0.77,0.87}
\definecolor{lightyellow}{rgb}{1.00,1.00,0.88}
\definecolor{limegreen}{rgb}{0.20,0.80,0.20}
\definecolor{linen}{rgb}{0.98,0.94,0.90}
\definecolor{magenta1}{rgb}{1.00,0.00,1.00}
\definecolor{magenta2}{rgb}{0.93,0.00,0.93}
\definecolor{magenta3}{rgb}{0.80,0.00,0.80}
\definecolor{magenta4}{rgb}{0.55,0.00,0.55}
\definecolor{magenta}{rgb}{1.00,0.00,1.00}
\definecolor{maroon1}{rgb}{1.00,0.20,0.70}
\definecolor{maroon2}{rgb}{0.93,0.19,0.65}
\definecolor{maroon3}{rgb}{0.80,0.16,0.56}
\definecolor{maroon4}{rgb}{0.55,0.11,0.38}
\definecolor{maroon}{rgb}{0.69,0.19,0.38}
\definecolor{mediumaquamarine}{rgb}{0.40,0.80,0.67}
\definecolor{mediumblue}{rgb}{0.00,0.00,0.80}
\definecolor{mediumorchid}{rgb}{0.73,0.33,0.83}
\definecolor{mediumpurple}{rgb}{0.58,0.44,0.86}
\definecolor{mediumsea}{rgb}{0.24,0.70,0.44}
\definecolor{mediumslate}{rgb}{0.48,0.41,0.93}
\definecolor{mediumspring}{rgb}{0.00,0.98,0.60}
\definecolor{mediumturquoise}{rgb}{0.28,0.82,0.80}
\definecolor{mediumviolet}{rgb}{0.78,0.08,0.52}
\definecolor{midnightblue}{rgb}{0.10,0.10,0.44}
\definecolor{mintcream}{rgb}{0.96,1.00,0.98}
\definecolor{mistyrose}{rgb}{1.00,0.89,0.88}
\definecolor{moccasin}{rgb}{1.00,0.89,0.71}
\definecolor{navajowhite}{rgb}{1.00,0.87,0.68}
\definecolor{navyblue}{rgb}{0.00,0.00,0.50}
\definecolor{navy}{rgb}{0.00,0.00,0.50}
\definecolor{oldlace}{rgb}{0.99,0.96,0.90}
\definecolor{olivedrab}{rgb}{0.42,0.56,0.14}
\definecolor{orange1}{rgb}{1.00,0.65,0.00}
\definecolor{orange2}{rgb}{0.93,0.60,0.00}
\definecolor{orange3}{rgb}{0.80,0.52,0.00}
\definecolor{orange4}{rgb}{0.55,0.35,0.00}
\definecolor{orangered}{rgb}{1.00,0.27,0.00}
\definecolor{orange}{rgb}{1.00,0.65,0.00}
\definecolor{orchid1}{rgb}{1.00,0.51,0.98}
\definecolor{orchid2}{rgb}{0.93,0.48,0.91}
\definecolor{orchid3}{rgb}{0.80,0.41,0.79}
\definecolor{orchid4}{rgb}{0.55,0.28,0.54}
\definecolor{orchid}{rgb}{0.85,0.44,0.84}
\definecolor{palegoldenrod}{rgb}{0.93,0.91,0.67}
\definecolor{palegreen}{rgb}{0.60,0.98,0.60}
\definecolor{paleturquoise}{rgb}{0.69,0.93,0.93}
\definecolor{paleviolet}{rgb}{0.86,0.44,0.58}
\definecolor{papayawhip}{rgb}{1.00,0.94,0.84}
\definecolor{peachpuff}{rgb}{1.00,0.85,0.73}
\definecolor{peru}{rgb}{0.80,0.52,0.25}
\definecolor{pink1}{rgb}{1.00,0.71,0.77}
\definecolor{pink2}{rgb}{0.93,0.66,0.72}
\definecolor{pink3}{rgb}{0.80,0.57,0.62}
\definecolor{pink4}{rgb}{0.55,0.39,0.42}
\definecolor{pink}{rgb}{1.00,0.75,0.80}
\definecolor{plum1}{rgb}{1.00,0.73,1.00}
\definecolor{plum2}{rgb}{0.93,0.68,0.93}
\definecolor{plum3}{rgb}{0.80,0.59,0.80}
\definecolor{plum4}{rgb}{0.55,0.40,0.55}
\definecolor{plum}{rgb}{0.87,0.63,0.87}
\definecolor{powderblue}{rgb}{0.69,0.88,0.90}
\definecolor{purple1}{rgb}{0.61,0.19,1.00}
\definecolor{purple2}{rgb}{0.57,0.17,0.93}
\definecolor{purple3}{rgb}{0.49,0.15,0.80}
\definecolor{purple4}{rgb}{0.33,0.10,0.55}
\definecolor{purple}{rgb}{0.63,0.13,0.94}
\definecolor{red1}{rgb}{1.00,0.00,0.00}
\definecolor{red2}{rgb}{0.93,0.00,0.00}
\definecolor{red3}{rgb}{0.80,0.00,0.00}
\definecolor{red4}{rgb}{0.55,0.00,0.00}
\definecolor{red}{rgb}{1.00,0.00,0.00}
\definecolor{rosybrown}{rgb}{0.74,0.56,0.56}
\definecolor{royalblue}{rgb}{0.25,0.41,0.88}
\definecolor{saddlebrown}{rgb}{0.55,0.27,0.07}
\definecolor{salmon1}{rgb}{1.00,0.55,0.41}
\definecolor{salmon2}{rgb}{0.93,0.51,0.38}
\definecolor{salmon3}{rgb}{0.80,0.44,0.33}
\definecolor{salmon4}{rgb}{0.55,0.30,0.22}
\definecolor{salmon}{rgb}{0.98,0.50,0.45}
\definecolor{sandybrown}{rgb}{0.96,0.64,0.38}
\definecolor{seagreen}{rgb}{0.18,0.55,0.34}
\definecolor{seashell1}{rgb}{1.00,0.96,0.93}
\definecolor{seashell2}{rgb}{0.93,0.90,0.87}
\definecolor{seashell3}{rgb}{0.80,0.77,0.75}
\definecolor{seashell4}{rgb}{0.55,0.53,0.51}
\definecolor{seashell}{rgb}{1.00,0.96,0.93}
\definecolor{sienna1}{rgb}{1.00,0.51,0.28}
\definecolor{sienna2}{rgb}{0.93,0.47,0.26}
\definecolor{sienna3}{rgb}{0.80,0.41,0.22}
\definecolor{sienna4}{rgb}{0.55,0.28,0.15}
\definecolor{sienna}{rgb}{0.63,0.32,0.18}
\definecolor{skyblue}{rgb}{0.53,0.81,0.92}
\definecolor{slateblue}{rgb}{0.42,0.35,0.80}
\definecolor{slategray}{rgb}{0.44,0.50,0.56}
\definecolor{slategrey}{rgb}{0.44,0.50,0.56}
\definecolor{snow1}{rgb}{1.00,0.98,0.98}
\definecolor{snow2}{rgb}{0.93,0.91,0.91}
\definecolor{snow3}{rgb}{0.80,0.79,0.79}
\definecolor{snow4}{rgb}{0.55,0.54,0.54}
\definecolor{snow}{rgb}{1.00,0.98,0.98}
\definecolor{springgreen}{rgb}{0.00,1.00,0.50}
\definecolor{steelblue}{rgb}{0.27,0.51,0.71}
\definecolor{tan1}{rgb}{1.00,0.65,0.31}
\definecolor{tan2}{rgb}{0.93,0.60,0.29}
\definecolor{tan3}{rgb}{0.80,0.52,0.25}
\definecolor{tan4}{rgb}{0.55,0.35,0.17}
\definecolor{tan}{rgb}{0.82,0.71,0.55}
\definecolor{thistle1}{rgb}{1.00,0.88,1.00}
\definecolor{thistle2}{rgb}{0.93,0.82,0.93}
\definecolor{thistle3}{rgb}{0.80,0.71,0.80}
\definecolor{thistle4}{rgb}{0.55,0.48,0.55}
\definecolor{thistle}{rgb}{0.85,0.75,0.85}
\definecolor{tomato1}{rgb}{1.00,0.39,0.28}
\definecolor{tomato2}{rgb}{0.93,0.36,0.26}
\definecolor{tomato3}{rgb}{0.80,0.31,0.22}
\definecolor{tomato4}{rgb}{0.55,0.21,0.15}
\definecolor{tomato}{rgb}{1.00,0.39,0.28}
\definecolor{turquoise1}{rgb}{0.00,0.96,1.00}
\definecolor{turquoise2}{rgb}{0.00,0.90,0.93}
\definecolor{turquoise3}{rgb}{0.00,0.77,0.80}
\definecolor{turquoise4}{rgb}{0.00,0.53,0.55}
\definecolor{turquoise}{rgb}{0.25,0.88,0.82}
\definecolor{violetred}{rgb}{0.82,0.13,0.56}
\definecolor{violet}{rgb}{0.93,0.51,0.93}
\definecolor{wheat1}{rgb}{1.00,0.91,0.73}
\definecolor{wheat2}{rgb}{0.93,0.85,0.68}
\definecolor{wheat3}{rgb}{0.80,0.73,0.59}
\definecolor{wheat4}{rgb}{0.55,0.49,0.40}
\definecolor{wheat}{rgb}{0.96,0.87,0.70}
\definecolor{whitesmoke}{rgb}{0.96,0.96,0.96}
\definecolor{white}{rgb}{1.00,1.00,1.00}
\definecolor{yellow1}{rgb}{1.00,1.00,0.00}
\definecolor{yellow2}{rgb}{0.93,0.93,0.00}
\definecolor{yellow3}{rgb}{0.80,0.80,0.00}
\definecolor{yellow4}{rgb}{0.55,0.55,0.00}
\definecolor{yellowgreen}{rgb}{0.60,0.80,0.20}
\definecolor{yellow}{rgb}{1.00,1.00,0.00}

\usepackage{amsmath}
\usepackage{soul}

\newcommand{\len}{\ell}

\newcommand{\ch}{\widehat{C}}
\newcommand{\Z}{{\mathbb Z}}
\newcommand{\R}{{\mathbb R}}

\newcommand{\N}{{\mathbb N}}

\newcommand{\Om}{\Omega}

\newcommand{\fkc}{B}
\newcommand{\cB}{\mathcal{B}}
\newcommand{\cC}{\mathcal{C}}

\newcommand{\cH}{\mathcal{H}}

\newcommand{\rad}{{\rm rad}}

\DeclareMathOperator{\diam}{diam}

\def\vint_#1{\mathchoice
	{\mathop{\vrule width 6pt height 3 pt depth -2.5pt
			\kern -8pt \intop}\nolimits_{#1}}%
	{\mathop{\vrule width 5pt height 3 pt depth -2.6pt
			\kern -6pt \intop}\nolimits_{#1}}%
	{\mathop{\vrule width 5pt height 3 pt depth -2.6pt
			\kern -6pt \intop}\nolimits_{#1}}%
	{\mathop{\vrule width 5pt height 3 pt depth -2.6pt
			\kern -6pt \intop}\nolimits_{#1}}}

\theoremstyle{plain}
\newtheorem{theorem}[equation]{Theorem}
\newtheorem{corollary}[equation]{Corollary}
\newtheorem{lemma}[equation]{Lemma}
\newtheorem{proposition}[equation]{Proposition}

\numberwithin{equation}{section}

\theoremstyle{definition}
\newtheorem{definition}[equation]{Definition}

\title[Sharp Hausdorff content estimates for accessible boundaries]{Sharp Hausdorff content estimates for accessible boundaries of domains in metric measure spaces of controlled geometry}
\author{Sylvester Eriksson-Bique, Ryan Gibara, Riikka Korte, and Nageswari Shanmugalingam }

\begin{document}
	
	\begin{abstract}
		We give a sharp Hausdorff
content estimate for the size of the  accessible boundary of any domain in a metric measure space of controlled geometry, i.e., 
a complete metric space equipped with a doubling measure supporting a $p$-Poincar\'e inequality for a fixed $1\le p<\infty$.
This answers a question posed by Jonas Azzam. In the process, we extend the result to every doubling gauge in 
metric measure spaces which satisfies a codimension one bound.
	\end{abstract}
	
	\maketitle
	
	
	\medskip

	\noindent {\small Mathematics Subject Classification (2020): {Primary:  28A78; Secondary: 30L15, 51F30.
		}

		\section{Introduction}
		
		Consider a domain (that is, an open and connected set) $\Omega\subset \R^d$ and a point $x_0 \in \Omega$. From both a geometric and potential-theoretic perspective, it is relevant to consider the portion of the boundary $\partial\Om$ which can be reached from $x_0$ by well behaved curves. Deemed the \emph{accessible boundary} of $\Omega$ with center $x_0$, for $c\geq 1$ the set $\partial_{x_0} \Omega(c)$ consists of all points $w\in \partial \Omega$ for which there exists a curve 
		$\gamma$ with end points $x_0$ and $w$ such that whenever $z$ is a point in the trajectory 
		of $\gamma$, we have $\ell(\gamma_{z,w})\le c\, d(z,\partial\Om)$. Here,
		$d(z,\partial\Om):=\inf\{d(z,y)\, :\, y\in\partial\Om\}$, and $\gamma_{z,w}$ stands in for \emph{every} subcurve of
		$\gamma$ with end points $z,w$. In the literature, this boundary has also been referred to as the \emph{visible boundary} of the domain. 
		
		Besides their natural geometric appeal, one is often interested in domains for which $\partial_{x_0} \Omega(c)$ is large for some $c\geq 1$. In measuring this size, the cardinality of the accessible boundary is not a good choice from the point of view of geometric measure theory. A sharper notion of size of a set $A$ is given by Hausdorff content:
		\begin{equation}\label{def:hcont}
			\mathcal{H}^s_r(A):=\inf\left\{\,\sum_{i =1}^\infty \diam(A_i)^s \, \colon\, 
			A \subset \bigcup_{i=1}^\infty A_i,\text{ and } \diam(A_i)\leq r \text{ for each }i\in\N\right\}.
		\end{equation}
		Estimates of the accessible boundary in terms of Hausdorff contents are important, for example, in establishing Hardy inequalities on domains \cite{ihlehrtuom, KosLehr09, Lehr14}. 
		
		The present note is concerned with whether a lower bound for the Hausdorff content of $\partial \Omega$ implies a lower bound for the Hausdorff content of $\partial_{x_0} \Omega(c)$ for some $c\geq 1$ and for all $x_0\in\Omega$. This has been answered in the affirmative by \cite{azzam,KNN} with Hausdorff-content exponent $0<s\leq d-1$ on $\partial\Om$ and Hausdorff-content exponent
$t$ with $0<t<s$ on $\partial_{x_0} \Omega(c)$. When $s=d-1$, then this is known to be the best that one can do -- see \cite[page 890]{azzam}. On the other hand, for $0<s<d-1$ it was asked by Azzam in \cite{azzam} whether one can improve the Hausdorff-content exponent of $\partial_{x_0} \Omega(c)$ to $t=s$. Our goal is to prove this sharp result.
		
		\begin{theorem}\label{thm:euclidean}
			Let $0<s<d-1$. For every $\widehat{C} > 1$ there exist constants $c, C \geq 1$ and a constant $\Delta>0$ so that any domain $\Omega \subset \R^d$ and any $\delta\in (0,1)$ satisfy the following property.
			
			If for all $x_0 \in \Omega$ we have that 
			\[
			\mathcal{H}^s_\infty\left(B(x_0,\ch d(x_0,\partial \Omega))\setminus \Omega\right)\geq \delta \, d(x_0,\partial\Om)^s,
			\] 
			then for all $x_0 \in \Omega$ we have that
			\[
			\mathcal{H}^s_\infty(\partial_{x_0} \Omega (c) \cap B(x_0,C\, d(x_0,\partial \Omega)))\geq \Delta\, \delta\,  d(x_0,\partial\Om)^s.
			\] 
		\end{theorem}
		
In fact, we prove a much more general version of this theorem for a metric measure space satisfying 
minimal assumptions and Hausdorff content that is measured with respect to a gauge function.

		In general metric spaces, and for complex sets, we may wish to measure Hausdorff content with respect to a general gauge function. A \emph{gauge function} $\xi$ associates to each open ball $B=B(x,r)$ a number $\xi(B) \in (0,\infty)$. 
The Hausdorff content $\mathcal{H}^\xi_{\infty}$ with respect to a gauge function is defined in Subsection \ref{subsec:hcont}. We will assume that $\xi$ is \emph{doubling}: there exists a constant $D_\xi$ so that for any pair of balls $B',B$ with $B' \cap B\neq\emptyset$ and $\rad(B')\leq \rad(B)\leq 2\rad(B')$, we have $D_\xi^{-1}\xi(B') \leq \xi(B)\leq D_\xi\, \xi(B')$. 
A distinguished doubling gauge is given by the 
standard codimension-$p$ gauge function $h_p(B)=\frac{\mu(B)}{\rad(B)^p}$ for some $1\leq p <\infty$. 
		
We say that $\xi$ has \emph{codimension at least $Q$}, if there is a 
constant $K\geq 1$, so that for all balls $B=B(x,r)$ and all radii $R\geq r$, we have
\begin{equation}\label{eq:codim-def}
	\frac{\xi(B(x,R))}{\xi(B(x,r))}\leq K 
	\frac{h_Q(B(x,R))}{h_Q(B(x,r))}.
\end{equation}
Examples of such gauges are given by $\xi=h_{Q+\sigma}$ for any $\sigma\ge 0$.
A reader uneasy with this generality may consider only gauges of this form. 
Another example is a gauge of the form
\[
\xi(B(x,R))=\frac{\mu(B(x,R))}{R^Q\, (\log(e+R))^q}
\]
for any $q>0$.
		
Codimensional Hausdorff measures such as $h_Q$, $Q>0$, play a crucial role in
the geometric analysis of subsets of metric measure spaces where the measure is not necessarily Ahlfors regular.
For instance, $h_1$ plays a key role in the study of measure-theoretic boundaries of sets of finite 
perimeter~\cite{AMP04, FSS03, HMM22, Lah20-B},
and $h_\theta$ for certain choices of $\theta>0$ plays a key role in understanding traces of Sobolev functions,
as for example in~\cite{BBS-hyp-fill, GibKorSh, GibSh2, Maly}.
More general gauge functions such as the last example above occur naturally in probability, for instance,
in Brownian motion~\cite[page~60]{Mattila}, see~\cite{KaenSuom07, Suomala08} for more examples.
		
		In~\cite{GibKort22}, under the hypothesis of Ahlfors regularity of the measure $\mu$,
a version of the following theorem was obtained, with the Hausdorff content used in the conclusion~\eqref{eq:conc-gauge} 
having exponent strictly less than that of the hypothesis~\eqref{eq:hyp-gauge}, with the gauge-related Hausdorff contents
replaced by Hausdorff contents. Subsequently, in~\cite{previouspaper}, the authors improved the results of~\cite{GibKort22} by expanding to
the case that $\mu$ is doubling but not necessarily Ahlfors regular. The gauge function considered there was
$\xi=h_p$ in~\eqref{eq:hyp-gauge}, and the conclusion was again weaker in that the gauge function in~\eqref{eq:conc-gauge}
was $h_q$ for some $q>p$.  
As in the Euclidean case, the sharp result with Hausdorff content exponents in both the hypothesis and 
conclusion being equal presented a technical obstruction.
In the present paper, we overcome this obstruction
at the expense of a more involved proof than what was done in~\cite{previouspaper}, 
leading to the following result. This solves the aforementioned 
open problem of Azzam, as given in Theorem~\ref{thm:euclidean} above in the Euclidean case, which now follows
from Theorem~\ref{thm:mainthm} with the use of the gauge function $\xi=h_{d-s}$.

To prove Theorem~\ref{thm:mainthm}, we developed a new technique based on studying an avoidance property
for curves, see Lemma~\ref{lem:ballpointsold} and Lemma~\ref{lem:iterationstep}; these two
lemmas are also of independent interest. Indeed, similar ideas have been useful elsewhere and can be found in~\cite[Theorem 1.1]{lahticapacity}, and in \cite[Section 4]{Panu}.

\begin{theorem}\label{thm:mainthm}
	Let $(X,d,\mu)$ be a complete metric measure space with $\mu$ a doubling measure supporting a 
$p$-Poincar\'e inequality for some 
$1\leq p<\infty$, and let $\xi$ be a doubling gauge function which has codimension strictly greater than $p$. For every $\widehat{C} > 1$ there exist constants $c_2, C_2 \geq 1$ and a constant $\Delta>0$ so that any domain $\Omega \subset X$ and any $\delta\in (0,1)$ satisfy the following property.
			
If for all $x_0 \in \Omega$ we have that
\begin{equation}\label{eq:hyp-gauge}
	\mathcal{H}^\xi_{d(x_0,\partial \Omega)}\left(B(x_0,\ch d(x_0,\partial \Omega))\setminus \Omega\right)\geq \delta\, \xi(B(x_0,d(x_0,\partial \Omega))),
\end{equation}
then for all $x_0 \in \Omega$ we have that
\begin{equation}\label{eq:conc-gauge}
	\mathcal{H}^\xi_{d(x_0,\partial \Omega)}\left(\partial_{x_0} \Omega (c_2) \cap B(x_0,C_2 d(x_0,\partial \Omega))\right)\geq \Delta\, \delta \,\xi(B(x_0,d(x_0,\partial \Omega))).
\end{equation}
\end{theorem}
		
We assume in Theorems \ref{thm:euclidean} and \ref{thm:mainthm} a content lower bound for the 
complement of $\Omega$. In contrast, the paper of Heinonen~\cite{HeiWall} and the papers of V\"ais\"al\"a~\cite{VWall-E, VWall}
consider the question of which domains in $\R^d$ satisfy the property
\[
\mathcal{H}^{d-1}_\infty(\partial \Omega \cap B(x_0,\ch d(x_0,\partial \Omega)))\geq \delta \,d(x_0,\partial\Om)^{d-1},
\] 
for each $x_0\in\Om$. The hypothesis in the first of our main theorems is equivalent to the above 
property when $s=d-1$. Further, since the accessible boundary is a subset of the topological boundary, we also obtain the following non-trivial fact.
		
\begin{corollary}\label{cor:boundarysize} 
	Let $(X,d,\mu)$ be a complete metric measure space with $\mu$ a doubling measure supporting a 
$p$-Poincar\'e inequality for some
$1\leq p <\infty$, and let $\xi$ be a doubling gauge function which has codimension strictly greater than $p$. For every $\widehat{C} \geq 1$ there exist a constant $C_2 \geq 1$ and a constant $\Delta>0$ so that any domain $\Omega \subset X$ and any $\delta\in (0,1)$ satisfy the following property.
			
If for all $x_0 \in \Omega$ we have that 
\[
	\mathcal{H}^\xi_{d(x_0,\partial \Omega)}\left(B(x_0,\ch d(x_0,\partial \Omega))\setminus \Omega\right)\geq \delta\, \xi(B(x_0,d(x_0,\partial \Omega))),
\] 
then for all $x_0 \in \Omega$ we have that
\[
\mathcal{H}^\xi_{d(x_0,\partial \Omega)}\left(\partial \Omega \cap B(x_0,C_2 d(x_0,\partial \Omega))\right)\geq \Delta \,\delta\, \xi(B(x_0,d(x_0,\partial \Omega))).
\] 
\end{corollary} 
	
In fact, this statement has a direct proof which is considerably simpler than that of Theorem \ref{thm:mainthm}, avoiding 
many of the technicalities. We will present it in Appendix \ref{sec:appendix}. Thus, a reader may wish to read the paper 
in the following order: first read the preparatory material in Sections \ref{sec:prelim} and  \ref{sec:preparation} followed 
by Appendix~\ref{sec:appendix} and only then read Section~\ref{sec:proof}. 
		
\medskip
		
\noindent {\bf Acknowledgement:} 
Nageswari Shanmugalingam's work is partially supported by the NSF (U.S.A.) grant DMS~\#2054960. Sylvester Eriksson-Bique's work was partially supported by the Finnish Academy grants no. 356861 and 354241. Part of the research was conducted during S.~E.-B.'s
visit to the University of Cincinnati in December 2022, and he wishes to thank that institution for its kind hospitality. We also thank Antti V\"ah\"akangas for helpful discussions, which helped in the writing of the paper.

\section{Preliminaries}\label{sec:prelim}
		
Let $(X,d,\mu)$ be a complete and separable metric measure space equipped with 
a Radon measure $\mu$. The open ball in $X$, centered at a point $x\in X$ and with
radius $r>0$, is denoted by $B(x,r)=\{y \in X: d(x,y)<r\}$. 
Each ball is considered to have an associated radius and center, and thus is formally a triple. 
Different radii and centers may define the same 
set, but we consider these as different balls. 
Given any ball $B=B(x,r)$, we refer to its (unique) radius as $\rad(B)=r$. The inflation/deflation of a ball is 
denoted by $CB:=B(x,Cr)$ for $C>0$. Curves are continuous maps $\gamma:I\to X$, where 
$I=[a,b]\subset \R$ is a compact interval. The length of rectifiable curves is denoted $\len(\gamma)$. 
The distance between a point $x\in X$ and a set $A\subset X$ is given by $d(x,A):=\inf_{a\in A} d(x,a)$.

The measure $\mu$ is said to be \emph{doubling} if there exists a constant $D_\mu\geq 1$ so that 
\[
	0<\mu(B(x,2r))\leq D_\mu\mu(B(x,r)) < \infty
\]
for every open ball $B(x,r)\subset X$. With $1\le p<\infty$,
the space $(X,d,\mu)$ is said to satisfy a $p$-Poincar\'e inequality if there exist constants $C_P,\lambda$ so that for all locally integrable functions $f\in L^1_{\rm loc}(X)$ and all upper gradients $g$ of $f$ we have
\begin{equation}\label{eq:piineq}
	\vint_{B(x,r)} |f-f_{B(x,r)}| d\mu \leq C_P r \left(\vint_{B(x,\lambda r)} g^p d\mu\right)^{1/p},
\end{equation}
where $f_A:=\vint_A f d\mu:= \frac{1}{\mu(A)} \int_A f d\mu,$ whenever the final integral is defined and $\mu(A)\in (0,\infty)$. A non-negative Borel function $g:X\to [0,\infty]$ is an upper gradient for $f$ if
\[
	|f(\gamma(0))-f(\gamma(1))|\leq \int_\gamma g \,ds
\] 
for all rectifiable curves $\gamma:[0,1]\to X$. For more background on upper gradients and analysis on metric spaces, 
see~\cite{heinonenkoskela}.
		
	\subsection{Notational conventions}\label{sub:convention}
		
Given two expressions $A$ and $B$, we write $A\lesssim B$ if there is a constant $C$ so that $A\leq C B$ with $C$ depending
only on the data associated with the metric measure space and the doubling constant associated with the
gauge functions used in computing Hausdorff contents. We also say that $A\simeq B$ if $A\lesssim B$ and $B\lesssim A$.
		
Throughout this paper $(X,d,\mu)$ will be a complete metric measure space, where $\mu$ is doubling and $X$ supports a 
$p$-Poincar\'e inequality. Such spaces $X$ satisfy two further properties: \emph{quasiconvexity} and 
\emph{the metric doubling property}, see
for instance~\cite[Theorem~8.3.2]{HKST}. A space $X$ is said to be $L$-quasiconvex if for each $x,y\in X$
there is a curve $\gamma$ with end points $x,y$ such that $\ell(\gamma)\le L\, d(x,y)$. We will repeatedly, and implicitly, use the fact that $d(x,X\setminus \Omega)\le L \, d(x,\partial \Omega)$ for $L$-quasiconvex spaces. We say that $X$ is $N_1$-metric doubling
for some $N_1\ge 1$ if for each $x\in X$ and $r>0$ there are at most $N_1$ number of points in $B(x,r)$ that are $r/2$-separated.
Observe that a metric space equipped with a $D$-doubling measure is necessarily $D^4$-metric doubling.

We denote the doubling constants associated with the gauge $\xi$ and the measure $\mu$ by
$D_\xi, D_\mu$ respectively. 
However, as their precise values are not kept track of in this work, we will replace $D_\xi$ and $D_\mu$ with 
$D=\max\{D_\xi,D_\mu\}$.

	\subsection{Hausdorff content}\label{subsec:hcont}
		
Let $\xi$ be a gauge function. If $\cC$ is a countable collection of open balls in $X$, we set  
$\rad(\cC):=\sup_{B\in \cC} \rad(B)$.
We also define
\[
	\xi(\cC):=\sum_{B\in \cC} \xi(B),
\]
$\bigcup \cC := \bigcup_{\fkc\in \cC} B$,
and for $C>0$ we set
$C\cC = \{C\fkc : \fkc \in \cC\}$. If $A\subset \bigcup \cC$, we say that 
$\cC$ is a cover of $A$, or that $\cC$ covers $A$. We further define the restricted collection $\cC|_A=\{B\in \cC : B \cap A \neq \emptyset\}$. 
It is worth noting that the two operations, restriction $\cC|_A$ and inflation $C\cC$, do not commute. In our notation we will use parenthesis to indicate the order of operations. Primarily we will use the order $(C\cC)|_A$, that is, we first inflate the balls and then restrict. 
		
Given a subset $A\subset X$ and $r\in (0,\infty]$,
we define the Hausdorff content at scale $r$ with respect to the gauge $\xi$ by 
\[
\cH^\xi_r(A):=\inf \{ \xi(\cC) \, \colon \, A\subset \bigcup \cC, \, \rad(\cC) \leq r, \, \text{and } \cC \text{ is countable}\}.
\]
		
The following lemma allows us to compare Hausdorff content at different scales. 
		
\begin{lemma}\label{lem:si-cont-scale-change}
Let $\tau\ge 1$ and suppose that $\xi$ is a doubling gauge function. Then there is a constant 
$C(\tau)>0$, depending on $D$, such that 
for each $K\subset X$ and $r>0$ we have 
\[
\cH^\xi_r(K)\le C(\tau)\, \cH^\xi_{\tau r}(K).
\]
\end{lemma}
		
\begin{proof} The argument is similar to \cite[Lemma 2.6]{previouspaper}.
Now, let $\cC$ be a covering of $K$ by balls with $\rad(\cC)\leq \tau r$. By iterating the metric doubling property, each of the balls in $\cC$ can be covered by $N'$ balls of radius $r$, for some fixed $N'\in \N$. Collect all of the resulting balls into a collection $\cC'$. Then, applying the doubling property of the $\xi$ we get a constant $C(\tau)$ so that 
\[
\cH^\xi_r(K)\leq \xi(\cC')\leq C(\tau) \xi(\cC).
\]
Infimizing over all such covers $\cC$ yields the claim.
\end{proof}
		
	At various points throughout the paper, we will need to compare $\cH^\xi_r$ when $\xi$ is a doubling gauge function of codimension at least $Q>0$ to $Q$-codimensional Hausdorff content, as well as the $Q$-codimensional Hausdorff content to the $p$-codimensional Hausdorff content for $0<p<Q$. The following lemma is the way we obtain these estimates.

\begin{lemma}\label{lem:conentconv} 
Suppose that $\xi$ is a doubling gauge function with codimension at least $Q>p>0$. There exists a constant $K_1$ so that the following holds. Let $B_0=B(x_0,r_0)$ be any ball, and let $\cB$ be a collection of balls with $\rad(\cB)\leq r_0$ and so that for every $B'\in \cB$ we have $B'\cap B(x_0,r_0) \neq \emptyset$. It then holds that 
\begin{equation}\label{eq:xisum}
	K_1\frac{\xi(\cB)}{\xi(B_0)}\geq \frac{h_Q(\cB)}{h_Q(B_0)} \geq \frac{h_p(\cB)}{h_p(B_0)}\frac{\rad(B_0)^{Q-p}}{\rad(\cB)^{Q-p}}.
\end{equation}
The constant $K_1$ depends only on the doubling constants of $\xi$ and 
the measure $\mu$, and the codimensionality constant $K$ in \eqref{eq:codim-def}.
\end{lemma}
		
\begin{proof} 
Let $x_B,r_B$ be the center and radius of any ball $B$. By the codimension assumption, see~\eqref{eq:codim-def}, we 
have for all $x\in X$ and all $R>r>0$ that
\begin{equation}
\frac{\xi(B(x,r))}{\xi(B(x,R))}\geq K^{-1} 
\frac{h_Q(B(x,r))}{h_Q(B(x,R))}.
\end{equation}
Thus, for any ball $B' \in \cB$ we get
\begin{equation}\label{eq:xihqcodim}
	\frac{\xi(B')}{\xi(B(x_{B'},3r_0))}\geq K^{-1} 
	\frac{h_Q(B')}{h_Q(B(x_{B'},3r_0)}.
\end{equation}
			
	We have $B(x_0,r_0)\subset B(x_{B'}, 3r_0) \subset B(x_0,5r_0)$. Thus, by doubling, we have $\xi(B(x_{B'},3r_0))\geq D^{-2}\xi(B(x_0,r_0))$ and $h_Q(B(x_{B'},3r_0)\leq D^{3}h_Q(B(x_0,r_0))$. Consequently, from these and \eqref{eq:xihqcodim} we get
\begin{equation}
	\frac{\xi(B')}{\xi(B_0)}\geq K^{-1} D^{-5} \,
	\frac{h_Q(B')}{h_Q(B_0)}.
\end{equation}
By summing over all the balls $B'\in \cB$ and by choosing $K_1=KD^{5}$ , we obtain the first part of the claimed inequality \eqref{eq:xisum}.
			
The second part of \eqref{eq:xisum} follows by noting that 
\[
\frac{h_Q(B')}{h_Q(B_0)}= \frac{h_p(B')}{h_p(B_0)} \frac{\rad(B_0)^{Q-p}}{\rad(B')^{Q-p}} \geq \frac{h_p(B')}{h_p(B_0)} \frac{\rad(B_0)^{Q-p}}{\rad(\cB)^{Q-p}}.\qedhere
\]
\end{proof}

\subsection{John domains}\label{subsec:john}
		
Recall that given $z\in\Om$, $d(z,\partial\Om)=\inf\{d(z,y)\, :\, y\in\partial\Om\}$.
		
\begin{definition}\label{def:johncurve} 
Let $\Omega \subset X$ be an open set, and $x \in \Omega$. A $c$-John curve  
$\gamma$ starting at $x$  is a rectifiable curve 
such that, with $w$ the other end point of $\gamma$,  
for each point $z$ in the trajectory of $\gamma$ we have that
\[
\ell(\gamma_{z,w})\le c\, d(z,\partial\Om).
\]
Here $\gamma_{z,w}$ denotes \emph{each} subcurve of $\gamma$ with end points $z,w$.  
\end{definition}
		
Given a sequence of points geometrically approaching the boundary, we can patch together a sequence of John curves to obtain a John curve from the initial point to the final point. In the limit, we obtain a John curve from the initial point to the boundary.

\begin{lemma}\label{lem:repeated curve} 
	Let $b \in (0,1), c_1\ge 1$. Then there exists $c_2\ge 1$ so that the following holds for all 
	$N\in\N\cup \{\infty\}$.
	Let $x_k \in \Omega$ 
	with $0\leq k < N, k\in \N$ and $x_N\in \overline{\Omega}$, where $\Omega \subset X$ is an open set. Suppose that  
$d(x_{k+1},\partial\Omega)\leq b\, d(x_{k},\partial\Omega)$ for every $k=0,1,\dots, N-1$, and further that 
$x_\infty=\lim_{k\to\infty}x_k$ if $N=\infty$.
			
	If for every $k=0,\dots, N-1$ there exists a $c_1$-John curve $\gamma_k:[0,1]\to X$ starting at $x_k$ with $\gamma_k(1)=x_{k+1}$, then there exists a $c_2$-John curve $\gamma:[0,1]\to X$ starting at 
$x_0$ with $\gamma(1)=x_N$.  This curve is 
obtained as the concatenation of the curves $\gamma_k$, $k=0,\ldots, N-1$.
\end{lemma}
		
\begin{proof}
Let $\gamma$ be the concatenation of the curves $\gamma_k$, $k=0,\ldots, N-1$. Note that
the length of $\gamma_k$ satisfies the estimate $\ell(\gamma_k)\le c_1\, d(x_k,\partial\Omega)$ because of the $c_1$-John condition.
Moreover, for each $k,j\in\{0,\ldots, N-1\}$ with $k<j$, we have that $d(x_j,\partial\Omega)\le b^{j-k}\, d(x_k,\partial\Omega)$. It follows that if 
$z$ is a point in the trajectory of $\gamma$, 
we fix a subcurve $\gamma_{z,x_N}$ of $\gamma$ with terminal points $z$ and $x_N$, and then choose the largest
$k\in\{1,\ldots, N-1\}$ such that $\gamma_{z,x_N}\subset \bigcup_{j=k}^{N-1}\gamma_j$. 
If $k=N-1$, then $\gamma_{z,x_N}$ is a subcurve of the $c_1$-John curve $\gamma_{N-1}$, and so we have
directly the inequality $\ell(\gamma_{z,x_N})\le c_1\, d(z,\partial\Om)$.
Now, if $k\le N-2$, then
\begin{align*}
\ell(\gamma_{z, x_N})\le \sum_{j=k+1}^{N-1}\ell(\gamma_j)+\ell((\gamma_k)_{z,x_{k+1}})
&\le c_1\sum_{j=k+1}^{N-1}d(x_j,\partial\Omega)+\ell((\gamma_k)_{z,x_{k+1}})\\
&\le c_1\sum_{j=k+1}^{N-1}b^{j-k-1}d(x_{k+1},\partial\Omega)+\ell((\gamma_k)_{z,x_{k+1}})\\
&\le \frac{c_1}{1-b}\, d(x_{k+1},\partial\Omega)+c_1\, d(z,\partial\Omega).
\end{align*}
If $d(x_{k+1},\partial\Omega)\le 2c_1\, d(z,\partial\Omega)$, then we have
\[
	\ell(\gamma_{z, x_N})\le\left[\frac{2c_1^2}{1-b}+c_1\right]\, d(z,\partial\Omega).
\]
If $d(x_{k+1},\partial\Omega)>2c_1\, d(z,\partial\Omega)$, then we would have
\[
	d(z,\partial\Omega)\ge d(x_{k+1},\partial\Omega)-d(x_{k+1},z)>2c_1\, d(z,\partial\Omega)-\ell((\gamma_k)_{x_{k+1},z})
	\ge c_1\, d(z,\partial\Omega),
\]
which is not possible as $c_1\ge 1$. Note here that we have used the $c_1$-John condition above as well. It follows that
necessarily $d(x_{k+1},\partial\Omega)\le 2c_1\, d(z,\partial\Omega)$, and so we have the John condition
\[
	\ell(\gamma_{z, x_N})\le\left[\frac{2c_1^2}{1-b}+c_1\right]\, d(z,\partial\Omega).
\]
Now taking 
\[
	c_2:=\max\bigg\lbrace \frac{2c_1^2}{1-b}+c_1, c_1\bigg\rbrace=\frac{2c_1^2}{1-b}+c_1
\]
yields the claim. 
\end{proof}

\subsection{Weak maximal function bounds}
		
Let $g\in L^1_{\rm loc}(X)$. For $s>0$ and $x\in X$, we define the 
restricted maximal function as
\[
M_s g(x):=\sup\left\{ \vint_{B(y,r)} |g| d\mu : x \in B(y,r), r \in (0,s)\right\}.
\]
		
We will need the following statement for this maximal function.
		
\begin{lemma} \label{lem:localweak} 
Suppose that $(X,d,\mu)$ is $D$-doubling. There exists a constant $C_M>0$ so that the following holds. 
If $g\in L^1_{\rm loc}(X)$ is non-negative, $x\in X$, and $r,s>0$, then
there exists $y\in B(x,r)$ so that
\[
	M_s g(y) \leq C_M\, \vint_{B(x,r+2s)} g d\mu.
\]
\end{lemma}
		
\begin{proof}
	Without loss of generality, we may assume that $\vint_{B(x,r+2s)} g d\mu$ is finite,
and we set $\Lambda:=\vint_{B(x,r+2s)} g d\mu$.
Let $\tilde g = g1_{B(x, r+2s)}$. Let $E_T = \{y \in X : M_s \tilde{g} > T\}$.  By the weak maximal function inequality \cite[Theorem~3.5.6]{HKST}, there exists a constant $C$, depending only on the doubling constant $D$ 
(strictly speaking, on $D_\mu$), such that
\[
\mu(E_T) \leq \frac{C}{T}\,\int_X \tilde g d\mu 
= \frac{C}{T}\,\int_{B(x,r+2s)} g d\mu = \frac{C}{T}\, \Lambda \mu(B(x,r+2s)).
\]
If $T=2C \Lambda\, D^{2+\lceil \log_2(s/r) \rceil_+}$, then we get from measure doubling that 
$\tfrac{C}{T}\,\Lambda \mu(B(x,r+2s)) < \mu(B(x,r))$. So,
\[
 \mu(E_T)< \mu(B(x,r)),
\]
and there exists a $y\in B(x,r)\setminus E_T$. For such $y$, we have $M_s \tilde{g}(y) 
\leq 2C \Lambda\,D^{2+\lceil \log_2(s/r) \rceil_+}$. We have $M_s \tilde{g}(y) =M_s g(y)$ since if $y\in B(z,\rho)$ 
for some $z\in X$ and $0<\rho<s$, 
then $B(z,\rho)\subset B(x,r+2s)$. The claim thus follows with $C_M=2C\, D^{2+\lceil \log_2(s/r) \rceil_+}$.
\end{proof}

\subsection{Pointwise Poincar\'e inequality}
		
\begin{definition}\label{def:PointwisePI}
We say that a space $X$ satisfies the pointwise $p$-Poincar\'e inequality with constants $C_P,L,\lambda\geq 1$, if for 
every $x,y\in X$ and every non-negative Borel function $g\in L^p(X)$, we have
\[
\inf_{\gamma\in \Gamma^{L,\lambda}_{x,y}} \int_\gamma g ds \leq C_P d(x,y)\left(M_{\lambda d(x,y)} g^p(x)+M_{\lambda d(x,y)} g^p(y)\right)^{\frac{1}{p}},
\]
where the infimum is taken over the collection $\Gamma^{L,\lambda}_{x,y}$ which consists of all rectifiable curves $\gamma$ connecting $x$ to $y$ with length at most $Ld(x,y)$ which are contained in $B(x,2\lambda d(x,y))$.
\end{definition}
		
The notion of pointwise Poincar\'e inequality is perhaps not a standard notion in current literature on analysis in nonsmooth spaces,
but the following proposition shows that it is equivalent to a notion that \emph{is} standard.
The proof of the proposition, as stated, can be found in~\cite[Theorem 1.5]{ebkeithzhong}, and the proof there uses
the Haj\l asz-type inequality~\cite[Theorem 3.2, Theorem~3.3]{hajkos}.
On geodesic spaces the $\lambda$ that appears in the Poincar\'e inequality \eqref{eq:piineq} can be chosen to be unity, 
$\lambda=1$, see \cite[Theorem 1]{hajkos95}. This observation together with \cite[Theorem 3.2, Theorem~3.3]{hajkos} yields the 
second paragraph of the proposition.
		
\begin{proposition}\label{prop:pointwise} 
Suppose that $(X,d,\mu)$ is doubling. The space $X$ satisfies a $p$-Poincar\'e inequality if and only if $(X,d,\mu)$ satisfies a pointwise $p$-Poincar\'e inequality.
			
Furthermore, if $X$ is a complete geodesic space that supports a $p$-Poincar\'e 
inequality, then the pointwise $p$-Poincar\'e inequality holds with the choice of $\lambda=2$,
at the expense of increasing the constant $C_P$. 
\end{proposition}

\section{Preparation for the proof of Theorem~\ref{thm:mainthm}}\label{sec:preparation}
		
\subsection{Reduction to geodesic spaces}\label{subsec:reduction}
		
First, we do a simple reduction of the problem to the geodesic case. A space $X$ is said to be geodesic if for every $x,y\in X$ there exists a rectifiable curve $\gamma:[0,1]\to X$ with $\gamma(0)=x,\gamma(1)=y$ and $\len(\gamma)=d(x,y)$. Every 
complete metric space equipped with a doubling measure that supports a $p$-Poincar\'e inequality
is bi-Lipschitz to a geodesic space, thanks to the fact that they are necessarily quasiconvex, see
for instance~\cite[Theorem~8.3.2]{HKST}. For such spaces, we can remove one parameter from the statement of 
Theorem~\ref{thm:mainthm} and it will simplify the proofs to do so. We now re-state Theorem~\ref{thm:mainthm} in this context.
		
\begin{theorem}\label{thm:mainthm-geodesic}
Let $(X,d,\mu)$ be a complete geodesic space with $\mu$ a doubling measure supporting a
$p$-Poincar\'e inequality. Let $\xi$ be a 
doubling gauge function which has codimension greater than $p$. There exist constants $c \geq 1$ and a constant $\Delta>0$ so that any open subset $\Omega \subset X$ and $\delta\in(0,1)$ satisfy the following property.
			
If for all $x \in \Omega$ it holds that 
\begin{equation}\label{eq:Haus-scl-ass}
	\mathcal{H}^\xi_{d(x,\partial \Omega)}(B(x,2 d(x,\partial \Omega))\setminus \Omega)\geq \delta\, \xi(B(x,d(x,\partial \Omega))),
\end{equation}
then for all $x \in \Omega$ we have
\begin{equation}\label{eq:Haus-scl-concl}
	\mathcal{H}^\xi_{d(x,\partial \Omega)}(\partial_x \Omega(c) \cap B(x,16 d(x,\partial \Omega)))\geq \Delta\, \delta\, \xi(B(x,d(x,\partial \Omega))).
\end{equation}
\end{theorem}
		
	Theorem~\ref{thm:mainthm} follows from this version. 
First note that even if $X$ is not a geodesic space, by virtue of the doubling property of $\mu$
and the support of a Poincar\'e inequality, $X$ is necessarily quasiconvex. Therefore, the length metric on $X$ is
biLipschitz equivalent to the original metric, and the length of curves in $X$ with respect to this new metric agrees
with the length in the original metric. The balls with respect  to the new metric are not necessarily balls in the original
metric, but they are contained in balls of comparable radii in the original metric, and contain balls in the original metric,
also of comparable radii. Thus modifying the metric results in a change of the relevant constants 
in Theorem~\ref{thm:mainthm}. Thus we can reduce the proof of Theorem~\ref{thm:mainthm} to the case that
$X$ is a geodesic space.
		
We now justify our claim that Theorem~\ref{thm:mainthm} follows from Theorem~\ref{thm:mainthm-geodesic} when
$X$ is a geodesic space. Suppose now that the hypothesis of Theorem~\ref{thm:mainthm} holds for some 
$\delta\in (0,1)$ and $\ch$ and for each $x\in\Om$. Let $x_0\in\Om$. Then by the geodesicity and completeness
of $X$ there is a nearest point to $x_0$ on $\partial\Om$ and a geodesic connecting that point to $x_0$, and except
for that point, the remaining part of the geodesic lies in $\Om$. Now we can find a point $x$ on this geodesic
that is a distance $d(x_0,\partial\Om)[1-1/\ch]$ from $x_0$. Note that then $d(x,\partial\Om)=d(x_0,\partial\Om)/\ch$.
Also, $B(x,\ch\, d(x,\partial\Om))\subset B(x_0,2d(x_0,\partial\Om))$. Hence, by the hypothesis, 
\[
\mathcal{H}^\xi_{d(x,\partial\Om)}( B(x_0,2d(x_0,\partial\Om)\setminus \Omega)
\ge \mathcal{H}^\xi_{d(x,\partial\Om)}( B(x,\ch\, d(x,\partial\Om))\setminus \Omega)\ge \delta\, \xi(B(x,d(x,\partial\Om))).
\]
Now, by the doubling property of $\xi$ and by Lemma~\ref{lem:si-cont-scale-change}, we have that
\[
\mathcal{H}^\xi_{d(x_0,\partial\Om)}( B(x_0,2d(x_0,\partial\Om)\setminus \Omega)
	\ge C(\ch)^{-1}\, \delta\, \xi(B(x,d(x,\partial\Om)))\ge C(\ch)^{-1}\, D^{-n_0}\, \, \delta\, \xi(B(x_0,d(x_0,\partial\Om))),
\]
where $n_0\simeq \log_2(\ch)$,
from whence the desired conclusion follows from an application of Theorem~\ref{thm:mainthm-geodesic}.

\subsection{An avoidance property}
		
Recall that $h_q(B)=\frac{\mu(B)}{\rad(B)^q}$ for $q>0$.  Given the reduction to the geodesic case, we will henceforth only state results for such spaces.
Recall the notational conventions from Subsection~\ref{subsec:hcont}. In particular, given a collection $\mathcal{F}$ of
subsets of $X$, we set $\bigcup\mathcal{F}:=\bigcup_{K\in\mathcal{F}}K$.
		
The following lemma coins a crucial avoidance property, which underlies our approach. It states that a collection $\cB$ satisfying 
a certain ``sparseness'' condition, \eqref{eq:Bojar}, can be avoided by a rectifiable curve. 
Conversely, if the collection of balls cannot be avoided, then \eqref{eq:Bojar} must fail. In fact, our application of Lemma~\ref{lem:ballpointsold} will use this contrapositive form.

\begin{lemma}\label{lem:ballpointsold} 
Under the hypotheses of Theorem~\ref{thm:mainthm-geodesic}, we denote the constants associated with the pointwise $p$-Poincar\'e
inequality by $(C_P,L,2)$.
There exist constants $C_1,\epsilon_0>0$ so that the following holds. Suppose that $x,y\in X$ and $\cB$ is a
countable collection of balls so that $\bigcup 4 \cB \cap \{x,y\} =\emptyset$, and so that for all $r\in (0,4d(x,y)]$ and 
$z\in \{x,y\}$ we have
\begin{equation}\label{eq:Bojar}
	h_p(\cB|_{B(z,r)}) \leq \epsilon_0 h_p(B(z,r)).
\end{equation}
Then, there exists a rectifiable curve $\gamma:[0,1]\to X$ connecting $x$ to $y$ which does not intersect 
$\bigcup\cB$ and with 
\begin{align*}
	\gamma(0)=x,\qquad & \gamma(1)=y,\qquad \\
	\gamma \subset B(x,2d(x,y)),  \qquad & \len(\gamma)\leq  C_1\,d(x,y), \\ 
		C_1\,d(\gamma(t),x)\geq  \len(\gamma|_{[0,t]}),\qquad & C_1 d(\gamma(t),y)\geq  \len(\gamma|_{[t,1]})
\end{align*}
for all $t\in [0,1]$.
\end{lemma}

\begin{proof} 
	Let $\epsilon_0>0$, and $C_1\geq 1$ be two fixed constants, to be chosen later. 
We will determine the constraints on their values in the course of the proof.  
Let $g=\sup_{B\in \cB} \frac{1}{\rad(B)} 1_{2B}.$
For each $r\in (0,2d(x,y))$ and for $z\in\{x,y\}$ we set 
$g_z:=\sup_{B\in \cB\, :\, B(z,r)\cap 2B\ne \emptyset} \frac{1}{\rad(B)} 1_{2B}$.
			
Noting that $g\vert_{B(z,r)}=g_z$, from~\eqref{eq:Bojar} and from the doubling proprty of $\mu$
we obtain
\begin{align*}
	\vint_{B(z,r)} g^p d\mu=\vint_{B(z,r)} g_z^p d\mu 
	&\le \frac{(2r)^p}{\mu(B(z,2r))}\, \frac{D}{2^p r^p}\, \sum_{B\in \cB\, :\, B(z,r)\cap 2B\ne \emptyset} \frac{\mu(2B\cap B(z,r))}{\rad(B)^p}\\
	&\le \frac{1}{h_p(B(z,2r))}\, \frac{D^{2}}{2^p r^p}\, \sum_{B\in \cB\, :\, B(z,r)\cap 2B\ne \emptyset} \frac{\mu(B)}{\rad(B)^p}\\
	&\le \frac{D^{2}\epsilon_0}{2^pr^p}.
\end{align*}
In applying ~\eqref{eq:Bojar} in the last line, we used the following fact: the non-emptiness of $2B\cap B(z,r)$ 
implies the non-emptiness of $B\cap B(z,2r)$ because $z\not\in 4B$.
			
	Let $\gamma_0:[0,d(x,y)]\to X$ be a geodesic connecting $x$ to $y$, parametrized by unit speed. Choose 
$a:=1-\frac{1}{18L}$. Let $x_0$ be the midpoint of this geodesic and let 
$x_i = \gamma_0(d(x,y)-a^i d(x,y)/2)$ for $i\geq 1$, and $x_i = \gamma_0(a^{-i} d(x,y)/2)$  
for $i\leq -1$. Let $r_i =(1-a)\, a^{|i|} d(x,y)/2$.
Note that with these choices, $r_i=d(x_i,z)(1-a)$ where $z=y$ if $i\ge 0$ and $z=x$ if $i<0$, and so
we have that
\[
\vint_{B(z,2d(z,x_i))}g^p\, d\mu\lesssim \frac{\epsilon_0}{[2d(z,x_i)]^p}. 
\]
As $B(x_i,18r_i)\subset B(z,2d(z,x_i))$,  
from our choice of $a$, we see that
\[
\vint_{B(x_i,18r_i)}g^p\, d\mu\le \frac{\mu(B(z,2d(z,x_i)))}{\mu(B(x_i,18r_i))}\, \vint_{B(z,2d(z,x_i))}g^p\, d\mu
\lesssim \left(\frac{2d(z,x_i)}{18r_i}\right)^Q\, \frac{\epsilon_0}{[2d(z,x_i)]^p}.
\]
Here $Q$ is the lower mass bound exponent inherited from iterating the doubling property of $\mu$; a choice of 
$Q=\log_2 D+1$ would suffice. As $d(z,x_i)=\frac{r_i}{1-a}$, it follows that 
\[
\vint_{B(x_i,18r_i)}g^p\, d\mu
	\lesssim \left(\frac{2d(z,x_i)}{18r_i}\right)^{Q-p}\, \frac{\epsilon_0}{r_i^p}
	\lesssim \frac{1}{(1-a)^{Q-p}}\, \frac{\epsilon_0}{r_i^p}\simeq  \frac{\epsilon_0}{r_i^p}.
\]
			
By Lemma \ref{lem:localweak}, with the choices of $r$ to be $r_i$ and $s$ to be $8r_i$, we know that 
there exists a sequence of points $y_i\in B(x_i,r_i)$ so that 
$M_{8 r_i} g^p(y_i) \lesssim\, \frac{\epsilon_0}{r_i^p}$.
Let $M$ denote the comparison constant from this comparison, so
\begin{equation}\label{eq:maxbounyi}
	M_{8 r_i} g^p(y_i) \le M\, \frac{\epsilon_0}{r_i^p}.
\end{equation} 
Furthermore, $d(y_i,z)\le r_i+d(x_i,z)=\left[1+\tfrac{1}{1-a}\right]r_i =(18L+1)r_{i}$.
			
Suppose that there is some $i\in\Z$ and $B\in\cB$ such that $y_i\in 2B$. We denote the center and radius of 
$B$ by $x_B$ and $r_B$ respectively. If $r_B<4r_i$, then 
\[
\frac{1}{(4r_i)^p}\le \frac{1}{r_B^p}\le M_{8r_i}g^p(y_i)\le M\, \frac{\epsilon_0}{r_i^p},
\]
from which we conclude that $\epsilon_0\ge 1/(4^p\, M)$. Hence, making sure that
\[
0<\epsilon_0<\frac{1}{4^p\, M},
\]
we know that such $B$ cannot exist. Therefore, we must have that $r_B\ge 4r_i$. 
Recall that $z\not\in 4B$. Then we have, for $z=y$ if $i \geq 0$ and $z=x$ otherwise, 
\[
	4r_B<d(x_B,z)\le d(z,y_{i})+d(y_{i},x_{B})=2r_B+(18L+1)r_i,
\]
from whence we obtain that necessarily $4r_i\le r_B\le 10L r_i$. Now as $y_i\in 2B$, we know from the geodesic property of
$X$ and by the fact that $r_B\ge 4r_i$ that $\mu(2B\cap B(y_i,r_i))\ge \mu(B(y_i,r_i))/D^2$. Therefore,
\[
\frac{1}{(10Lr_i)^p\, D^2}\le \frac{1}{D^2\, r_B^p}\le \frac{1}{r_B^p}\frac{\mu(2B\cap B(y_i,r_i))}{\mu(B(y_i,r_i))}\le M_{8r_i}g^p(y_i)\le M\, \frac{\epsilon_0}{r_i^p}.
\]
This is again not possible if 
\[
	0<\epsilon_0<\frac{1}{(10L)^p\, M\, D^2}.
\]
Thus, making sure to choose 
\[
	0<\epsilon_0<\min\bigg\lbrace\frac{1}{4^p\, M},\,\frac{1}{(10L)^p\, M\, D^2}\bigg\rbrace,
\]
we obtain that for each $B\in\cB$ the points $y_i\not\in 2B$ whenever $i\in\Z$.
			
Note that we assume that the pointwise $p$-Poincar\'e inequality holds with $\lambda=2$, see Definition~\ref{def:PointwisePI}
for the statement of this inequality. Note also that as $r_{i+1}-ar_i$ or $r_{i+1}=r_i/a$, we have
\[
	d(y_i,y_{i+1})\le d(y_i,x_i)+d(x_i,x_{i+1})+d(x_{i+1},y_{i+1})<\min\{2r_i+r_{i+1},2r_{i+1}+r_i\}
	\le \frac{2a+1}{a}\, \min\{r_i,r_{i+1}\},
\]
and so $\lambda d(y_i,y_{i+1})< 8\min\{r_i,r_{i+1}\}:=R_i$ because $1>a\ge 1/2$. Now,
invoking~\eqref{eq:maxbounyi} and using
the pointwise $p$-Poincar\'e inequality, there exist curves 
$\gamma_i\subset B(y_i,4d(y_i,y_{i+1}))$
connecting $y_i$ to $y_{i+1}$ with $\len(\gamma_i) \leq L d(y_i,y_{i+1})$ and 
\[
\int_{\gamma_i} g ds \le C_P\, d(y_i,y_{i+1})[M_{R_i}g^p(y_i)+M_{R_i}g^p(y_{i+1})]
\le C_P\, \frac{2a+1}{a}r_i\, \left[\frac{M\epsilon_0}{R_i^p}+\frac{M\epsilon_0}{R_{i}^p}\right]^{1/p}\le 
8C_P\, (2M)^{1/p}\epsilon_0^{1/p}.
\]
We restrict the choice of $\epsilon_0$ further as follows:
\[
0<\epsilon_0<\min\bigg\lbrace\frac{1}{4^p\, M},\,\frac{1}{(10L)^p\, M\, D^2},\, \frac{1}{2(8C_P)^p\, M}\bigg\rbrace,
\]
and see that then $\int_{\gamma_i}g\, ds<1$.
Thus, $\gamma_i$ does not intersect any of the balls $B\in \cB$. This is because otherwise,
$\gamma_i$ would intersect such $B$ \emph{and} at the same time have its end points $y_i,y_{i+1}$ outside the ball $2B$,
and so the integral of $g$ would be at least $1$. 
The desired curve $\gamma$ connecting $x$ to $y$ is obtained by concatenating all the curves $\gamma_i$ to each other. 
			
We now verify that the concatenated curve satisfies the listed properties claimed in the statement of the
lemma. For each integer $i$ we have that $\gamma_i\subset B(y_i,4d(y_i,y_{i+1}))$. If $w\in B(y_i,4d(y_i,y_{i+1}))$, then
\begin{align*}
	d(w,x)\le d(w,y_i)+d(y_i,x_i)+d(x_i,x)&<4d(y_i,y_{i+1})+r_i+d(y,x)\\
	&<4\, \min\{2r_i+r_{i+1},r_i+2r_{i+1}\}+r_i+d(y,x).
\end{align*}
Using the construction of $r_i$, $i\in\Z$, we see that
\[
	d(w,x)\le \begin{cases} 5r_i+8r_{i+1}+d(x,y)\le \left(\frac{13}{36L}+1\right)\, d(x,y) &\text{ if }i\ge 0,\\
	9r_i+4r_{i+1}+d(x,y)<\left(\frac{17}{36L}+1\right)\, d(x,y) &\text{ if }i<0,\end{cases}
\]
and so $d(w,x)<2\, d(x,y)$. Thus we conclude that $\gamma_i\subset B(x,2d(x,y))$, and hence $\gamma\subset B(x,2d(x,y))$.
Moreover, as 
\[
\ell(\gamma_i)\le Ld(y_i,y_{i+1})\le \tfrac{2a+1}{a}\, \min\{r_i,r_{i+1}\}\, L
\le 16L\, \frac{1-a}{2}\, d(x,y)\, a^{|i|},
\]
we have that $\ell(\gamma)\le 16\, L\, d(x,y)$. Next, for a point $w\not\in\{x,y\}$ in the trajectory of $\gamma$, note that 
there is some $i_w\in\Z$ such that $w$ is in the trajectory of $\gamma_{i_w}$. We break the rest of the proof down into
two cases, $i_w\le0$ and $i_w>0$.  Note that there may be more than one choice of $i_w$, and we want to
consider each of them. Thus for each such $i_w\leq 0$ we consider $\gamma_{x,w}$ to
represent  each subcurve of $\gamma$ with
terminal points $x,w$ such that $\gamma_{x,w}\subset \bigcup_{i=-\infty}^{i_w}\gamma_i$, and proceed. (And for $i_w>0$, we enforce $\gamma_{w,y}\subset \bigcup_{i=i_w}^{\infty}\gamma_i$.)
			
When $i_w\le0$, we have that 
\[
\ell(\gamma_{x,w})\le \sum_{i=-\infty}^{i_w}\ell(\gamma_i)\le 16L\, \frac{1-a}{2}\, d(x,y)\sum_{i=-\infty}^{i_w}a^{-i}
=8Ld(x,y)\, a^{-i_w}.
\]
On the other hand,
\begin{align*}
	d(x,w)\ge d(x,x_{i+1})-d(x_{i+1},y_{i+1})-d(y_{i+1},w)
	&= a^{-i_w-1}\, \frac{d(x,y)}{2}\, \left[ 1-(1-a)-16L(1-a)a\right]\\
	&\ge \frac{1}{18}\, d(x,y)\, a^{-i_w}.
\end{align*}
Therefore, we have the inequality
$\ell(\gamma_{x,w})\le 144L\, d(x,w)$.
By symmetry, if $i_w\ge 0$, we have that 
$\ell(\gamma_{w,y})\le 144L\, d(w,y)$. 
			
Next, if $i_w> 0$, then $\ell(\gamma_{x,w})\le \ell(\gamma)\le 16L\, d(x,y)$, and 
\begin{align*}
d(x,w)\ge d(x,x_i)-d(x_i,y_i)-d(y_i,w)&=\left[d(x,y)-a^{i_w}\frac{d(x,y)}{2}\right]-(1-a)\, a^{i_w}\frac{d(x,y)}{2}
		-\ell(\gamma_{i_w})\\
&\ge d(x,y)-\frac{d(x,y)}{2} \, a^{i_w}\, \left[2-a+16L(1-a)\right]\\
&\ge d(x,y)-a\, d(x,y)\\
&=\frac{1}{18L}\, d(x,y).
\end{align*}
We used the fact that $a^{i_w}\le a$. It follows that $\ell(\gamma_{x,w})\le 282L^2\, d(x,w)$. By symmetry,
we also obtain that when $i_w<0$, we have $\ell(\gamma_{w,y})\le 282L^2\, d(w,y)$. This completes the proof
once we have chosen $C_1=282\, L^2$.
\end{proof}

\section{Proof of Theorem \ref{thm:mainthm-geodesic}}\label{sec:proof}
		
The proof of Theorem \ref{thm:mainthm-geodesic} will be based on a contradiction argument involving iteration. 
Let $\cC$ be a countable collection of open balls. We say that a ball $B(x,r)$, with $x\in \Omega$ and $r>0$, 
is $\epsilon$-bad with respect to $\cC$ if 
\[
\xi(\cC|_{B(x,16r)})=\sum_{\fkc \in \cC|_{B(x,16r)}} \xi(\fkc)\leq \epsilon\,\xi(B(x,r)),
\] 
and $\rad(\cC|_{B(x,16r)})\leq r$. 
In our application of the following lemma to the proof of Theorem \ref{thm:mainthm-geodesic}, we will use $\cC$ to be a cover of
the accessible boundary of $\Omega$ from $x_0$, but the next lemma will be true even with a general choice of collection $\cC$. 
		
		The following Lemma is the basic step of our proof by contradiction. We briefly describe its content and how we use it. The lemma states that an $\epsilon$-bad ball within a domain satisfying the assumptions of Theorem \ref{thm:mainthm-geodesic}, with radius equal to the distance to the boundary, will satisfy one of two properties:
\begin{enumerate}
	\item either there exists a smaller $\epsilon$-bad ball, which is accessible by a John curve, 
and with radius equal to the distance to the boundary; or
	\item the collection that we started from failed to cover a certain accessible boundary.
\end{enumerate}
If we start from the assumption that $\cC$ is a covering of the accessible boundary, both of these are bad in a different way. In the second case, we actually failed to cover the desired boundary, which yields a direct contradiction. In the first case, when applied infinitely many times, we obtain a sequence of bad balls geometrically approaching the boundary, which fail to be well covered by $\cC$. The centres are connected by John curves, and Lemma \ref{lem:repeated curve} yields a limit point in the accessible boundary. The radius constraint on $\cC$ prevents this collection from covering the limit point.
		
\begin{lemma}\label{lem:iterationstep} 
Fix $\tau\ge 200$. There exist $c>1$ and $\Delta>0$
so that the following property holds.
Suppose that $\Omega$ is a proper domain, $x_0\in \Omega$,  $r_0=d(x_0,\partial  \Omega)$,  and $\xi$ is a 
doubling gauge function which, together, satisfy the inequality~\eqref{eq:Haus-scl-ass} for some $\delta\in(0,1)$.
			
If $\cC$ is a countable
collection of open balls with $\rad(\cC)\leq r_0$ and if $B_0=B(x_0,r_0)$ is an $\epsilon$-bad ball with
respect to $\cC$,  and with $\epsilon:=\Delta \delta$, then, either there exists a point 
\begin{equation}\label{eq:first-ste}
	x_1 \in B(x_0, 2r_0) \cap \partial_{x_0} \Omega(c) \setminus \bigcup \tau\cC,
\end{equation}
or there exists an $\epsilon$-bad ball, with respect to $\cC$, denoted $B_1=B(x_1,r_1)$, with 
$r_1=d(x_1,\partial\Omega) \leq r_0/641$ and with $x_1 \in B(x_0,8r_0)$ and so that $x_1$ can be 
reached from $x_0$ by a $c$-John curve.
\end{lemma}
		
\begin{proof} 
Let $\delta>0$ be the parameter given by the hypothesis~\eqref{eq:Haus-scl-ass}.
			
Without loss of generality, assume that for all $\fkc \in \cC$ we have $\fkc \cap B(x_0,16r_0)\neq \emptyset$. 
We have that $(X,d,\mu)$ satisfies a pointwise Poincar\'e inequality with constants $(C_P,L,2)$. 
In addition to the listed constants from Subsection~\ref{sub:convention},
there will be several parameters chosen in the proof. 
Their precise values are not important, however the order is significant, which is why we list them here, in the order 
in which they are determined, so that consistency of the choices is easier to observe.
\begin{enumerate}
\item Fix $\tau\geq 200$. We choose some $k\in \N$ such that $\tau \leq 2^k$ and $k\ge 10$. 
\item Parameters $\epsilon_0, C_1$ come from Lemma \ref{lem:ballpointsold}.
\item By assumption $\xi$ has codimension greater than $p$, that is codimension 
at least $Q$ with $Q>p$. Let $K_1$ be the corresponding constant from Lemma \ref{lem:conentconv}.                                      
\item Choose a density parameter 
\begin{equation}\label{eq:eta}
	\eta \in \left(0,\ \left(\frac{5}{36}\right)^{Q-p}\frac{\epsilon_0}{2K_1\, D^5}\right).  
\end{equation}
\item Choose the parameter from the statement of the lemma as 
\begin{equation}\label{eq:Cylon}
	\epsilon =\Delta \delta \in \left(0,\ \min\left\{D^{-20}\delta,\ \frac{\delta}{2\, C(\tau)\, D^k},\ 
	\frac{\delta  \eta}{4\, C(640)\, D^9}\right\}\right).
\end{equation}
This ensures 
\begin{equation}\label{eq:Delta}
	\Delta \in \left(0,\ \min\left\{D^{-20},\ C(\tau)^{-1}D^{-k}/2,\  \frac{\eta}{4\, C(640)\, D^9}\right\}\right).
\end{equation}
Note that $\Delta$ on its own does not play a role in this lemma, but will be needed in the
proof of the main theorem. For the convenience of the reader we added the condition for $\Delta$ here.
The constants $C(\tau),C(640)\ge 1$ are from Lemma~\ref{lem:si-cont-scale-change}.
\item Then, fix an expansion parameter 
\begin{equation}\label{eq:S}
	S>\max\left\{64, 1+\left(\frac{2K_1\eta D^{14}}{\epsilon\, \epsilon_0}\right)^{1/(Q-p)}\right\}.
\end{equation}
\item Based on $S$ and $C_1$, choose the John-curve parameter as $c\geq 1600 S C_1$.
\end{enumerate}
			
Let $E=B(x_0,2r_0) \setminus\left(\Omega\cup \bigcup \tau \cC \right)$. Suppose that~\eqref{eq:first-ste} is false.
In other words, for every $x\in E$, there is no $c$-John curve from $x_0$
which reaches $x$. Note that for some $x\in E$ this may simply be due to the fact that $x$ lies outside  
$\overline{\Omega}$. Our goal is to find an $\epsilon$-bad ball $B_1=B(x_1,r_1)\subset\Om$ with 
$r_1=d(x_1,\partial \Omega)\leq r_0/2$ and $x_1\in B(x_0,8r_0)$
so that $x_1$ can be reached from $x_0$ by a $c$-John curve based at $x_0$.
Given the initial setup, we now move to remove from $E$ a set of points where the ``density'' of  $\cC$ is too 
large, and to show that the resulting set has large density.
			
By ~\eqref{eq:Haus-scl-ass} we have  
$\cH^\xi_{r_0}(B(x_0,2r_0)\setminus \Omega)\geq \delta\, \xi(B(x_0,r_0))$, and thus
\[
\delta \xi(B(x_0,r_0))\le \cH^\xi_{r_0}(B(x_0,2r_0)\setminus \Omega)
\le \cH^\xi_{r_0}(E)+\cH^\xi_{r_0}(\bigcup\tau\cC).
\]
By Lemma~\ref{lem:si-cont-scale-change} and by the $\epsilon$-badness of $B(x_0,r_0)$, we know that
\[
\cH^\xi_{r_0}\left(\bigcup\tau\cC\right)\le C(\tau)\, \cH^\xi_{\tau r_0}\left(\bigcup\tau\cC\right)
\le C(\tau)\, \xi(\tau\cC)\le C(\tau)\, D^k\, \xi(\cC)
<C(\tau)\, D^k\, \epsilon \, \xi(B(x_0,r_0)).
\]
By~\eqref{eq:Cylon}, we know that $\epsilon < C(\tau)^{-1}D^{-k}\delta/2$, and so
\begin{equation}\label{eq:est-delta/2}
\frac{\delta}{2}\, \xi(B(x_0,r_0))\le (\delta-C(\tau)D^k\, \epsilon)\, \xi(B(x_0,r_0))<\cH^\xi_{r_0}(E).
\end{equation}
			
Next, we remove the high density points of $\cC$. Let 
\[
H_\eta :=\{x\in E: \exists r\in (0,64r_0), \xi((18\cC)|_{B(x,r)}) \geq \eta \xi(B(x,r))\}.
\]
We now show that $H_\eta$ has small content, more specifically,
that~\eqref{eq:Heta-est} holds. This will be done by finding a suitable cover of 
$H_\eta$. If $x\in E$  we have $x\notin \bigcup \tau \cC$. Let $r>0$.
If there is some $B\in\cC$ such that $18B\cap B(x,r)$ is non-empty, then denoting 
the center and radius of $B$ by $x_B$ and $r_B$,
we have that $\tau r_B<d(x,x_B)\le r+18r_B$, from which we see that 
\begin{equation}\label{eq:radiusbound}
	r>(\tau-18)r_B.
\end{equation}
As $\tau\geq 200>18$, it follows that
$r_B<r/(\tau-18)$. Hence, for each $x\in H_\eta$ there is some $0<r_x<64r_0$ such that 
$\xi((18\cC)|_{B(x,r_x)})\ge \eta\, \xi(B(x,r_x))$. 
By~\eqref{eq:radiusbound}, necessarily $\rad((18\cC)|_{B(x,r_x)})<18r_x/(\tau-18)$.
The collection $\{B(x,2r_x)\, :\, x\in H_\eta\}$ forms a cover of $H_\eta$,
and so by the $5$-covering theorem we can extract a countable pairwise disjoint subcollection $\{B(x_i,2r_i)\}_i$ of this cover
such that $\{B(x_i,10r_i)\}_i$ covers $H_\eta$.
			
We claim that if $B_1\in (18\cC)|_{B(x_i,r_i)}$ and $B_2\in (18\cC)|_{B(x_j,r_j)}$ with $j\ne i$, 
then $B_1$ and $B_2$ are disjoint. To see this, note that by the fact that $\rad(B_1)<18r_i/(\tau-18)$ and $\tau\geq 128$,
necessarily $B_1\subset B(x_i,2r_i)$. Similarly $B_2\subset B(x_i,2r_j)$. As $B(x_i,2r_i)$ and $B(x_j,2r_j)$ are disjoint, the claim
follows.
			
Now, by Lemma~\ref{lem:si-cont-scale-change}, 
\begin{align*}
	\cH^\xi_{r_0}(H_\eta)\le C(640)\, \cH_{640r_0}^\xi(H_\eta)\le C(640)\, \sum_i\xi(B(x_i,10r_i))
	&\le C(640)\, D^4\, \sum_i\xi(B(x_i,r_i))\\
	&\le C(640)\, D^4\, \eta^{-1}\, \sum_i\xi((18\cC)|_{B(x_i,r_i)})\\
	&\le C(640)\, D^4\, \eta^{-1}\, \xi(18\cC)\\
	&\le C(640)\, D^9\, \eta^{-1}\, \xi(\cC)\\
	&\le C(640)\, D^9\, \eta^{-1}\, \epsilon\, \xi(B(x_0,r_0)).
\end{align*}
Hence, the restriction on $\epsilon$ given in~\eqref{eq:Cylon} by
\[
\epsilon<\frac{\delta  \eta}{4\, C(640)\, D^9}
\]
yields that 
\begin{equation}\label{eq:Heta-est}
	\cH^\xi_{r_0}(H_\eta)<\tfrac{\delta}{4}\, \xi(B(x_0,r_0)),
\end{equation} 
and so by~\eqref{eq:est-delta/2} and sub-additivity of Hausdorff content, we have that
\begin{equation}\label{eq:EminHeta-est}
	\frac{\delta}{4}\, \xi(B(x_0,r_0))<\cH^\xi_{r_0}(E\setminus H_\eta).
\end{equation}
			
At this juncture, we have obtained a ``large'' set $E\setminus H_\eta$, which is not covered by $\tau\cC$. 
None of these points can be reached by a $c$-John curve from $x_0$. Fix then any $x\in E\setminus H_\eta$. 
Let $S\geq 1$ be as in~\eqref{eq:S} and let $C_1\geq 1$ be the constant from 
Lemma~\ref{lem:ballpointsold}. The point $x$ can not be reached by a $c$-John curve from $x_0$. 
			
We will now examine certain curves that connect $x_0$ to $x$ and show that each one of them must approach very closely to the 
boundary of the domain before reaching $x$. Let $\gamma:[0,1]\to B(x,2d(x_0,x))$ be any curve 
of length $C_1 d(x,x_0)$ 
connecting $x_0$ to $x$ and with $C_1d(\gamma(t),x_0)\geq  \len(\gamma|_{[0,t]})$ and 
$C_1d(\gamma(t),x)\geq  \len(\gamma|_{[t,1]})$ for all $t\in [0,1]$. We call these \emph{the four curve conditions}. 
As $X$ is a geodesic space, there is at least one such curve $\gamma$, namely the geodesic
curve with end points $x_0$ and $x$.
Since $x\in E$, no such curve  $\gamma$ 
is $c$-John with respect to $x_0$. Consequently, for every such $\gamma$, 
there exists some $t\in [0,1]$ so that 
\begin{equation}\label{eq:tcond}
	d(\gamma(t),X\setminus \Omega)\leq \frac{1}{c}\len(\gamma|_{[t,1]})\leq \frac{C_1}{c}d(\gamma(t),x).
\end{equation}
Choose the smallest $t^*$ for which the first inequality above holds. 
First, note that $t^*>0$, for when $t$ is small enough so that
$d(x_0,\gamma(t))<[1-2C_1/c]\, d(x_0,\partial\Om)$ and $\gamma(t)\in\Om$, we have that 
\[
d(\gamma(t),X\setminus\Om)=d(\gamma(t),\partial\Om)\ge d(x_0,\partial\Om)-d(x_0,\gamma(t))
>\tfrac{2C_1}{c}\, d(x_0,\partial\Om)\ge \tfrac1c\, \ell(\gamma)\ge \tfrac1c\, \ell(\gamma|_{[t,1]}).
\]
Next, note that $t^*<1$, for if not, then for each $t<1$ we have that
\[
d(\gamma(t),X\setminus \Omega)> \frac{1}{c}\len(\gamma|_{[t,1]}),
\]
and then letting $t\to 1^-$ we obtain that $\gamma$ is actually a $c$-John curve from $x_0$, which is not possible.
Hence, we have that $0<t^*<1$. Observe also that $\gamma(t^*)\in \Om$, for if not, then we must have that
$d(\gamma(t^*),X\setminus\Om)=0$, but then by the minimality of $t^*$, for each $0<t<t^*$ we have that 
$d(\gamma(t),X\setminus\Om)>\tfrac1c\ell(\gamma|_{[t,1]})$, and letting $t\to (t^*)^-$ we obtain that
$\ell(\gamma|_{[t^*,1]})=0$, i.e., $t^*=1$, again a contradiction. Thus $\gamma(t^*)\in \Om$.
Next, for each $0\le t<t^*$ we have that
\[
d(\gamma(t),X\setminus \Omega)> \frac{1}{c}\len(\gamma|_{[t,1]})\ge \frac{1}{c}\len(\gamma|_{[t,t^*]}).
\]
It follows that $\gamma|_{[0,t^*]}$ is a $c$-John curve from $x_0$ to $\gamma(t^*)$.
			
By ensuring that $c\geq 1600S C_1$, we obtain for the ball $B_\gamma:=B(z,s)$ 
with $z=\gamma(t^*)$ and $s=d(\gamma(t^*),\partial \Omega)>0$ that $s\leq d(x,z)/(1600S)$. 
Further, as $x\in E\subset B(x_0,2r_0)$, we have that
$d(x,z)\leq 8 r_0$, and thus $s< (641)^{-1}r_0$.
As noted, the curve $\gamma|_{[0,t^*]}$ is $c$-John and connects $x_0$ to $z$. 
			
Such a ball $B_\gamma$, or a fixed inflation of it, must intersect a ball in one of two collections which we now define:
\begin{align*}
	\cB_x := \{B(\widehat{z},s):\ & \widehat{z}\in \Omega\cap B(x_0, 8r_0), s=d(\widehat{z},\partial \Omega)
	< \min\{641^{-1}r_0, d(x,\widehat{z})/(1600S)\} \text{ and }\\
	&z \text{ can be reached by a $c$-John curve from } x_0 \text{ and } \rad(\cC|_{B(\widehat{z},16s)})\leq s\},
\end{align*} 
\begin{align*}
		\cC_b := \{\fkc \in \cC: \, \exists \widehat{z}\in \Omega, s&=d(\widehat{z},\partial \Omega)
		< \min\{641^{-1}r_0, d(x,\widehat{z})/(1600S)\} \text{ and }\\
		&\, \fkc\cap B(\widehat{z},16s)\neq \emptyset, \rad(\fkc)> s\}.
\end{align*}  
Then, with $z=\gamma(t^*)$, we have that either $B_\gamma=B(z,s)\in \cB_x$,
or else there is some $B\in\cC$ such that $B$ intersects $B(z,16s)$ and $\rad(B)>s$. 
See Figure \ref{fig:curves} 
for a depiction of the settings and alternatives just discussed.
			
\begin{figure}[h!]\includegraphics[width=.8\textwidth]{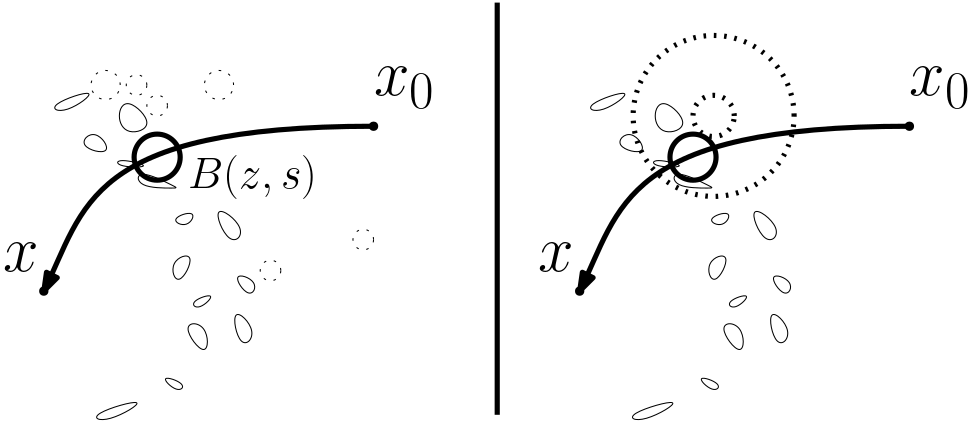}
				\caption{A figure depicting a curve $\gamma$ connecting $x$ to $x_0$, which satisfies the four curve conditions. Since the curve is not $c$-John, it approaches the boundary of the region (depicted as the collection of small obstacle sets between $x$ and $x_0$) too closely. The ball $B_\gamma=B(\gamma(t^*),s)$, which was constructed is shown as bolded, and some balls from $\cC_b$ are shown dashed. On the left hand side all of the balls in $\cC_b$ are quite small, and we have that $B_\gamma\in \cB_x$. On the right, the ball $B_\gamma$ intersects a ball in $\cC_b$ which is big, and thus is contained in its inflation. }\label{fig:curves}
			\end{figure}

Return now to our curve $\gamma$ and the ball $B_\gamma$.  If $B_\gamma\in \cB_x$, then $\gamma$ intersects 
$\bigcup \cB_x$. If $B_\gamma\not\in \cB_x$, then the only condition from the definition of $\cB_x$ which can fail 
is the last one. That is, $\rad(\cC|_{16B_\gamma})> s$, and there exists a $\fkc \in \cC$ with 
$\fkc\cap 16B_\gamma\neq\emptyset$ and $\rad(\fkc)> s$. That is, $\fkc \in \cC_b$. Thus, 
$B_\gamma \subset 18B$, and $\gamma$ so intersects a ball $18B$ for some $\fkc \in \cC_b$. In 
either case, every $\gamma$ which satisfies the four curve conditions intersects $\bigcup \cB_x \cup 18\cC_b$. 
Our goal next is to use this fact to deduce that the balls in $\cB_x$ have large content, and to 
apply a pidgeonholing argument to deduce that one of the balls in the collection $\cB_x$ is the $\epsilon$-bad ball which we are looking for, see Figure \ref{fig:blocking}. We will first need some preliminary work and to move to a disjoint subcollection.
	
\begin{figure}[h!]
		{\includegraphics[width=.4\textwidth]{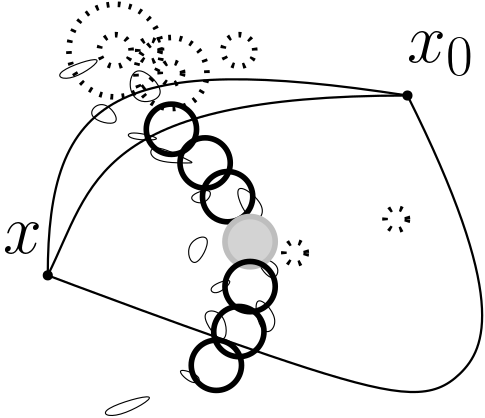}}
		\caption{A figure depicting some curves connecting $x$ and $x_0$, which satisfy the four curve conditions. All of these curves are blocked by a frontier formed by inflations of balls from $\cB_x$ (depicted as bolded balls), and some inflated balls in $\cC_b$, which are bolded and dotted. There are only relatively few balls in $\cC_b$, which is a consequence of the fact that $x\not\in H_\eta$. The argument of the proof gives a lower bound for the content of the balls coming from $\cB_x$. This uses the blocking property in conjunction with Lemma \ref{lem:ballpointsold}, and the fact that $x\not\in H_\eta$. Then, a pidgeonholing argument shows that one of the balls in $\cB_x$ is $\epsilon$-bad. In the figure, the $\epsilon$-bad ball is shaded in gray.} \label{fig:blocking}
	\end{figure}

Let $32\cB_x'\subset 32\cB_x$ be a disjoint collection so that $\bigcup \cB_x \subset \bigcup 160\cB_x'$. Every curve which satisfies the four curve conditions intersects some ball in $\cB_x \cup 18\cC_b$ and thus some ball in $160\cB_x' \cup 18\cC_b$.
Observe that as $32\cB_x'$ is a pairwise disjoint collection of open balls, necessarily it is a countable
collection. Moreover, as $\cC$ is a countable collection by hypothesis, therefore $\cC_b$ is also countable.
We will now aim to apply Lemma~\ref{lem:ballpointsold} to the countable collection
$\cB:=160\cB_x'\cup 18\cC_b$. However, to do so, we need to verify 
\begin{equation}\label{eq:intersection}
\{x,x_0\}\cap \left(\bigcup 640 \cB'_x \cup \bigcup 72\cC_b\right) = \emptyset.
\end{equation}
	
	Since every $B(z,s)\in \cB_x$ satisfies $s=d(z,\partial \Omega)< \min\{641^{-1}r_0, d(x,z)/(1600S)\}$, we have $\{x,x_0\}\cap \bigcup 640 \cB_x = \emptyset$. Since $640\cB_x'\subset 640\cB_x$, we get $\{x,x_0\}\cap \bigcup 640 \cB_x' = \emptyset$.  
To show that $\{x,x_0\}\cap \bigcup 72\cC_b=\emptyset$ it suffices to prove that 
$\{x,x_0\}\cap \bigcup72\cC = \emptyset$. Because $\tau \geq 200$ and so 
$72\cC_b\subset \tau\cC$, it follows from the fact that $x\in E\subset B(x_0,2r_0)\setminus\tau\cC$ that 
$x\not\in \bigcup 72 \cC_b \subset \bigcup \tau \cC$. So, 
we only need to check that $x_0\not\in \bigcup72\cC_b$. 
	
We now show the stronger conclusion that 
\begin{equation} \label{eq:cbballs}
B(x_0,r_0/2) \cap \bigcup 72 \cC_b =\emptyset.
\end{equation}
Indeed, if there was some $\fkc \in \cC_b$ so that $B(x_0,r_0/2)\cap 72B\neq \emptyset$, then
\[ 
r_0=d(x_0,\partial\Om)\le 
	\frac{r_0}{2}+72r_B+d(x_B,z)+d(z,\partial\Om)
	<\frac{r_0}{2}+90r_B,
\]
where $B=B(x_B,r_B)$. Hence $\tfrac{r_0}{180}<r_B\le r_0$. Therefore, 
\[
	\xi(B(x_0,r_0))\le D^8\xi(B(x_0,180r_B)\le D^9\xi(B(x_B,180r_B))\le D^{17}\xi(B).
\]
Since $\epsilon<D^{-20}<D^{-17}$, it follows that
$\xi(\fkc)\ge \epsilon \xi(B(x_0,r_0))$, which violates the assumption that $B(x_0,r_0)$ is
$\epsilon$-bad.
We thus get the disjointness of $B(x_0,r_0/2)$ and $\bigcup 72\cC_b$. This completes the proof 
of~\eqref{eq:cbballs} and hence of~\eqref{eq:intersection}. 
	
Since every curve satisfying the four curve conditions intersects a ball in $\cB$, 
we see that the conclusion of Lemma~\ref{lem:ballpointsold} fails, and hence
the hypotheses of that lemma must also fail. By construction, we know that $\cB$ does not intersect $x,x_0$. 
Therefore, \eqref{eq:Bojar} fails for either $z=x$ or for $z=x_0$ and for some $0<s\le 4d(x,x_0)$. Hence,
for some $\rho\in (0,8r_0)$ so that for either $z=x$ or for $z=x_0$,
\begin{equation}\label{eq:bsum}
	h_p((160\cB_x' \cup 18\cC_b)|_{B(z,\rho)}) \geq \epsilon_0 h_p(B(z,\rho)).
\end{equation}
If we are forced to have $z=x_0$ in~\eqref{eq:bsum}, then as $d(x,x_0)<2r_0$, we have
$B(x_0,\rho)\subset B(x,2r_0+\rho)\subset B(x_0,4r_0+\rho)$. We now claim that in this situation 
$\rho\ge r_0/2$. Indeed, if not, then
by~\eqref{eq:cbballs} we have that $18\cC_b\cap B(x_0,\rho)=\emptyset$, and, moreover, if 
$B=B(w,s)\in\cB_x$, then for $y\in B(w,160s)$ we know that 
\[
	d(y,x_0)\ge d(w,x_0)-160s\ge d(x_0,\partial\Om)-d(w,\partial\Om)-160s=r_0-161s>r_0-161\tfrac{r_0}{641}>\tfrac{r_0}{2},
\]
from which we also conclude that $160\cB_x'\cap B(x_0,\rho)=\emptyset$, violating~\eqref{eq:bsum}. If $z=x_0$, 
then necessarily $\rho\ge \tfrac{r_0}{2}$, proving the claim. So, 
$B(x_0,\rho)\subset B(x,2r_0+\rho)\subset B(x_0,9\rho)$. Therefore, by~\eqref{eq:bsum},
\begin{align*}
	h_p(B(x,2r_0+\rho))=\frac{\mu(B(x,2r_0+\rho))}{(2r_0+\rho)^p}
	&\le \frac{\mu(B(x_0,9\rho))}{(9\rho)^p}\, \left(\frac{9\rho}{2r_0+\rho}\right)^p\\
	&\le D^5\, \left(\frac45\right)^p\, h_p(B(x_0,\rho))\\
	&\le D^5\, \frac{1}{\epsilon_0}\, h_p((160\cB_x' \cup 18\cC_b)|_{B(x_0,\rho)}) \\
	&\le D^5\,  \frac{1}{\epsilon_0}\, h_p((160\cB_x' \cup 18\cC_b)|_{B(x,2r_0+\rho)}).
\end{align*}
Hence, we now have the existence of some $\rho'=2r_0+\rho\in (0,10r_0)$ such that 
\begin{equation}\label{eq:bsum-2}
	h_p((160\cB_x' \cup 18\cC_b)|_{B(x,\rho')})\ge D^{-5}\epsilon_0\, h_p(B(x,\rho')).
\end{equation} 
Because $x\not\in H_\eta$, we have that for each $r\in (0,64r_0)$,
\[
\xi((18\cC_b)_{B(x,r)})<\eta\, \xi(B(x,r)).
\]
Applying Lemma~\ref{lem:conentconv} to the choice of $r=\rho'$, together with the fact
that $\rho\ge r_0/2$ and so $\rho'\ge 5r_0/2$, we obtain
\[
	\eta>\frac{h_p((18\cC_b)|_{B(x,\rho')})}{h_p(B(x,\rho'))}\, \frac{1}{K_1}\, \left(\frac{\rho'}{\rad((18C_b)|_{B(x,\rho')})}\right)^{Q-p}
\ge \frac{h_p((18\cC_b)|_{B(x,\rho')})}{h_p(B(x,\rho'))}\, \frac{1}{K_1}\, \left(\frac{5}{36}\right)^{Q-p},
\]
and so by~\eqref{eq:eta},
\[
\frac{h_p((18\cC_b)_{B(x,\rho')})}{h_p(B(x,\rho'))}<\eta K_1\left(\frac{36}{5}\right)^{Q-p}
<\frac12D^{-5}\epsilon_0,
\]
whence from~\eqref{eq:bsum-2}, together with subadditivity, we see that 
\begin{equation}\label{eq:bsum-3}
		\frac{\epsilon_0}{2D^5}\, h_p(B(x,\rho'))<h_p((160\cB_x')|_{B(x,\rho')}).
\end{equation}
If, on the other hand, we had $z=x$ in~\eqref{eq:bsum}, then we also have the validity of~\eqref{eq:bsum-2},
and so the above argument using $x\not\in H_\eta$ yields the validity of~\eqref{eq:bsum-3} in both cases $z=x$ 
and $z=x_0$.
	
If $B\in\cB_x'$ such that $160B\cap B(x,\rho')\ne\emptyset$, then by the construction of $\cB_x\supset\cB_x'$, we 
know that $\rad(\cC|_{16B})\le \rad(B)$ and $\bigcup\cC|_{16B}\subset 18B$. Since $32\cB_x'$ is pairwise 
disjoint, we have that the family $\{\cC|_{16B}\}_{B}$ is pairwise disjoint as $B$ ranges over $32\cB_x'$.
	
Note also that for $B\in\cB_x$, the distance between $x$ and the center of $B$ is larger than $1600S\, \rad(B)$, and so
if in addition $160B\cap B(x,\rho')$ is nonempty, then we must have $1600S\, \rad(B)<\rho'+160\rad(B)$,
and hence 
\begin{equation}\label{eq:rho-prime}
	\rho'>(1600S-160)\, \rad(B) \text{ and }\rad(160\cB_x'|_{B(x,\rho')})\le \frac{\rho'}{10S-1}.
\end{equation}
Now, if $160B\in (160\cB_x')|_{B(x,\rho')}$ and $B'\in\cC|_{16B}$, then by the above arguments, we know that
$\rad(B')\le \rad(B)\le \tfrac{\rho'}{160(10S-1)}$. Moreover, for $w\in B'$, 
\[
d(x,w)<\rho'+160\,\rad(B)+16\,\rad(B)+2\,\rad(B')<\frac{160(10S-1)+178}{160(10S-1)}\rho'\le 2\rho'.
\]
Therefore, $\bigcup\cC|_{16B}\subset B(x,2\rho')$.
	
Now, as $x\not\in H_\eta$, and as $2\rho'<64r_0$, we see that
\[
\eta \xi(B(x,2\rho'))\ge \xi((18\cC)|_{B(x,2\rho')})=\sum_{B\in(18\cC)|_{B(x,2\rho')}}\xi(B)
\ge \sum_{B\in\cB_x'}\xi(\cC|_{16B}).
\]
An application of Lemma~\ref{lem:conentconv}, \eqref{eq:rho-prime},  and~\eqref{eq:bsum-3} gives
\[
\frac{\epsilon_0}{2D^5}\le \frac{h_p((160\cB_x')|_{B(x,\rho')})}{h_p(B(x,\rho'))}
\le \frac{K_1}{(10S-1)^{Q-p}}\, \frac{\xi((160\cB_x')|_{B(x,\rho')})}{\xi(B(x,\rho'))}.
\]
From the above two sets of inequalities, we obtain
\begin{align*}
\sum_{B\in (160B_x')|_{B(x,2\rho')}}\xi(\cC|_{\tfrac{1}{10}B})\le 
\sum_{B\in\cB_x'}\xi(\cC|_{16B})&\le \eta\, \xi(B(x,2\rho'))
\le  \frac{K_1}{(10S-1)^{Q-p}}\frac{2D^5}{\epsilon_0}\, D\, 
\xi((160\cB_x')|_{B(x,2\rho')})\\
&= \frac{K_1}{(10S-1)^{Q-p}}\frac{2D^5}{\epsilon_0}\, D\, 
\sum_{B\in(160B_x')|_{B(x,2\rho')}}\xi(B).
\end{align*}
Hence, by pidgeon-holing two of the sums in the previous equation, we obtain a ball $B\in\cB_x'$ 
such that $160B\cap B(x,2\rho')\ne\emptyset$ and
\[
\xi(\cC|_{16B})\le \frac{K_1}{(10S-1)^{Q-p}}\frac{2D^6}{\epsilon_0}\xi(160B)
\le  \frac{K_1}{(10S-1)^{Q-p}}\frac{2D^6}{\epsilon_0}\, D^8\, \xi(B).
\]
By the choice of $S$ from~\eqref{eq:S}, we now see that 
\[
\xi(\cC|_{16B})<\epsilon \, \xi(B),
\]
that is, $B$ is an $\epsilon$-bad ball that satisfies all the conditions listed in the conclusion of the lemma.
\end{proof}

We are now ready to prove Theorem~\ref{thm:mainthm-geodesic}. For an illustration of the proof, see Figure \ref{fig:finalarg}.

\begin{figure}[h!]\includegraphics[width=.8\textwidth]{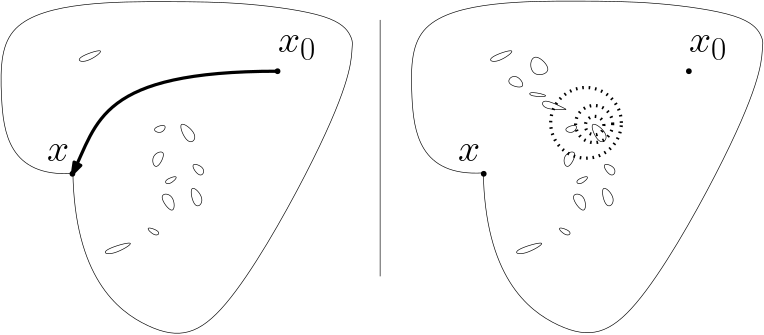}
	\caption{A figure depicting the final proof of Theorem \ref{thm:mainthm-geodesic}. The idea is to repeatedly use Lemma \ref{lem:iterationstep}. At each step of the argument, we face a dichotomy: either we have a $c$-John curve which connects $x_0$ to an $x$ which is not covered by $\cC$, \emph{or} we have a smaller $\epsilon$-bad ball (depicted as a dashed ball on the right). The first case is depicted on the left of the figure, and the second case is depicted on the right. The idea in the second case is to repeat the argument to obtain a geometrically converging sequence of $\epsilon$-bad balls, whose limit point can be reached by a $c_2$-John curve. This leads to a contradiction.}\label{fig:finalarg}
\end{figure}

\begin{proof}[Proof of Theorem~\ref{thm:mainthm-geodesic}] 
Let $c>1$, and let
$\Delta>0$ be the constant from Lemma~\ref{lem:iterationstep} and 
$\epsilon=\Delta \delta$. Let $c_2\geq 4c$ be the constant from Lemma~\ref{lem:repeated curve} 
corresponding to $b=1/2$ and $c_1=c$. 
	
Let $\Omega$ be a domain such that 
for all  $x\in \Omega$ with $r=d(x,\partial \Omega)$ the inequality~\eqref{eq:Haus-scl-ass} holds, i.e., 
\[
\cH^\xi_{r}(B(x,2r)\setminus \Omega)\geq \delta\, \xi(B(x,r)).
\]
Let $x_0\in\Om$. If
\begin{equation}\label{eq:Eeek}
\cH^\xi_{r_0}(\partial_{x_0} \Omega(c_2) \cap B(x_0,16 r_0))
\ge \epsilon\, \xi(B(x_0,r_0))=\Delta \,\delta\, \xi(B(x_0,r_0)),
\end{equation}
we have the conclusion of the theorem  with $c$ replaced with $c_2$. 
So now assume that this inequality fails. Set $r_0=d(x_0,\partial\Om)$.
Then there exists a countable 
cover $\cC$ of  $\partial_{x_0} \Omega(c_2) \cap B(x_0,16 r_0)$ with open balls with $\rad(\cC)\leq r_0$ 
and with $\xi(\cC)< \epsilon \,\xi(B(x_0,r_0))$.
Note that $\cC$ necessarily covers $\partial_{x_0} \Omega(c) \cap B(x_0,16 r_0)$ as well. Then $B(x_0,r_0)=:B_0$
is an $\epsilon$-bad ball with respect to the countable collection $\cC$. Now an application of
Lemma~\ref{lem:iterationstep} gives the existence of $B_1=B(x_1,r_1)$ which is also $\epsilon$-bad with respect
to the cover $\cC$ such that $0<r_1<r_0/641$, $x_1\in B(x_0,8r_0)$, and $x_1$ can be reached from $x_0$ 
by a $c$-John curve. This is because $\partial_{x_0} \Omega(c) \cap B(x_0,16 r_0)\subset \bigcup \cC$ 
and hence~\eqref{eq:first-ste}
cannot hold for any point. We also have, from the $\epsilon$-badness of $B_1$, that $\rad(\cC|_{B(x_1,16r_1)})\le r_1$.
	
We also see that $B(x_1,2r_1)\cap\partial_{x_1}\Om(c)\subset B(x_0,16r_0)\cap\partial_{x_0}\Om(c_2)$ by
Lemma~\ref{lem:repeated curve}. Hence we can, by 
repeating the above argument with $B_0$ replaced by $B_1$ and with the same cover $\cC$, obtain
a ball $B_2=B(x_2,r_2)$ with $0<r_2<r_1/641$, $x_2\in B(x_1,8r_2)$, with $x_2$ reachable from $x_1$ by a
$c$-John curve and $B_2$ is $\epsilon$-bad with respect to the cover $\cC$. Moreover,
$\rad(\cC|_{B(x_2,16r_2)})\le r_2$ by the $\epsilon$-badness of $B_2$.
	
Inductively, we obtain a sequence of points $x_i$ and radii $r_i$ satisfying the following conditions:
\begin{enumerate}
	\item $r_i\leq \tfrac{1}{641} r_{i-1}$,
	\item $x_i\in B(x_{i-1}, 8 r_{i-1})$,
	\item $r_i=d(x_i,\partial \Omega)$,
	\item $B(x_i,r_i)$ is an $\epsilon$-bad ball with $\rad(\cC|_{B(x_i,16r_i)})\le r_i$,
	\item $x_i$ can be connected to $x_{i-1}$ by a $c$-John curve starting from $x_{i-1}$,
	\item $B(x_i,2r_i)\cap\partial_{x_i}\Om(c)\subset B(x_0,16r_0)\cap\partial_{x_0}\Om(c_2)$.
\end{enumerate} 
The sequence $(x_i)_i$ is Cauchy in $X$ and hence has a limit point $x_\infty\in \overline{\Om}$. By
Lemma~\ref{lem:repeated curve} again we have that $x_\infty$ can be reached by a $c_2$-John curve from $x_0$.
Moreover, 
\[
d(x_0,x_\infty)<\sum_{j=0}^\infty d(x_j,x_{j+1})\le 8r_0\, \sum_{j=0}^\infty (641)^{-j}\le 9r_0,
\]
and so $x_\infty\in B(x_0,16r_0)\cap \partial_{x_0}\Om(c_2)$. If $x_\infty\in\bigcup\cC$, then there would be some
open ball $B\in\cC$ such that $x_\infty\in B$, and so for sufficiently large $i$ we must have that 
$B(x_i,r_i)\subset B$ (because $\lim_ir_i=0$), and this violates the condition that $\rad(\cC|_{B(x_i,16r_i)})\le r_i$.
Hence $x_\infty$ is not in $\bigcup\cC$, violating the choice of $\cC$ as a cover of 
$\partial_{x_0} \Omega(c_2) \cap B(x_0,16 r_0)$. Therefore, the assumption stated above cannot hold true,
and so~\eqref{eq:Eeek} must be true.
\end{proof}

\appendix
	
\section{Direct proof of Corollary~\ref{cor:boundarysize}}\label{sec:appendix}
	
In contrast to Theorem \ref{thm:mainthm}, which gives a lower bound for the accessible boundary, Corollary \ref{cor:boundarysize} only discusses the topological boundary. This yields considerable simplifications for a direct proof. The statements and arguments also become a little more transparent, which is why we present the fleshed out argument in the appendix.
	
First, we do the reduction to the geodesic case.
	
\begin{theorem}\label{thm:mainthm-geodesic-bdry}
Let $(X,d,\mu)$ be a complete geodesic space with $\mu$ a doubling measure supporting
a $p$-Poincar\'e inequality for some $1\le p<\infty$. Let $\xi$ be a 
doubling gauge function which has codimension strictly greater than $p$. There exist constants 
$c \geq 1$ and a constant $\Delta>0$ so that any open subset $\Omega \subset X$ and 
$\delta\in(0,1)$ satisfy the following property.
		
If for all $x \in \Omega$ it holds that 
\begin{equation}\label{eq:Haus-scl-ass-bdry}
\mathcal{H}^\xi_{d(x,\partial \Omega)}(B(x,2 d(x,\partial \Omega))\setminus \Omega)\geq \delta\, \xi(B(x,d(x,\partial \Omega))),
\end{equation}
then for all $x \in \Omega$ we have
\begin{equation}\label{eq:Haus-scl-concl-bdry}
\mathcal{H}^\xi_{d(x,\partial \Omega)}(\partial \Omega \cap B(x,16 d(x,\partial \Omega)))\geq \Delta\, \delta\, \xi(B(x,d(x,\partial \Omega))).
\end{equation}
\end{theorem}

The above theorem holds even in the case $p=1$ with $\xi=h_1$, but the proof differs from the one below. This case follows from the
so-called boxing inequality, see~ \cite{Pugilistic}.
	
The proof that Theorem~\ref{thm:mainthm-geodesic-bdry}, which assumes that $X$ is a geodesic space,
is equivalent to Corollary~\ref{cor:boundarysize}, which does not assume geodesicity, is mutatis mutandis the proof
that Theorem~\ref{thm:mainthm-geodesic} is equivalent to Theorem~\ref{thm:mainthm}, and so we omit it here.

\begin{proof}[Proof of Theorem \ref{thm:mainthm-geodesic-bdry}]
We fix $\tau>200$ and choose $k\ge 10$ such that $2^k>\tau$.
For simplicity, we set $\epsilon:=\delta\, \Delta$.
Suppose that we have a point $x_0\in\Om$ such that with $r_0=d(x_0,\partial\Om)$, we have
\begin{equation}\label{eq:AA-1}
	\mathcal{H}^\xi_{r_0}(B(x_0,2r_0)\setminus\Om)\ge \delta\, \xi(B(x_0,r_0))
\end{equation}
but
\begin{equation}\label{eq:AA-2}
	\mathcal{H}^\xi_{r_0}(B(x_0,4r_0)\cap\partial\Om)<\Delta\,\delta\,\xi(B(x_0,r_0))=\epsilon\, \xi(B(x_0,r_0)).
\end{equation}
Then we can find a countable collection $\mathcal{B}$ of balls covering $B(x_0,4r_0)\cap\partial\Om$ such that 
$\rad(\mathcal{B})<r_0$, $B\cap B(x_0,4r_0)\cap\partial\Om\ne \emptyset$  for each $B\in\mathcal{B}$, and
\begin{equation}\label{eq:AA-3}
	\xi(\mathcal{B})<\epsilon\, \xi(B(x_0,r_0)).
\end{equation}
With our choice of $\tau>200$ and $k\ge 10$ such that $2^k>\tau$, we also have that
\[
\xi(\tau\mathcal{B})<D^k\,\epsilon\, \xi(B(x_0,r_0)).
\]
Choosing $\Delta$ small enough, we make sure that 
\begin{equation}\label{eq:AA-Limit1}
D^k\Delta<\frac{1}{4},
\end{equation}
so that 
\begin{equation}\label{eq:AA-4}
	\mathcal{H}^\xi_{r_0}(B(x_0,2r_0)\setminus (\Om\cup\bigcup\tau\mathcal{B}))>\frac{3}{4}\, \delta\, \xi(B(x_0,r_0)).
\end{equation}
For each $x\in B(x_0,2r_0)\setminus (\Om\cup\bigcup\tau\mathcal{B})$, any curve in $B(x_0,4r_0)$ connecting $x$ to $x_0$
must intersect $\partial\Om\cap B(x_0,4r_0)$, and hence some ball in $\mathcal{B}$. Therefore 
by Lemma~\ref{lem:ballpointsold}, the hypotheses of that lemma must fail for the collection $\mathcal{B}$ and $x_0,x$.
		
If there is some $B\in\mathcal{B}$ such that $x_0\in 4B$, then as $B$ also intersects $\partial\Om$, with $x_B$ the center and
$r_B$ the radius of $B$, we must have $\frac{r_0}{4}\le r_B\le r_0$, and $d(x_0,x_B)<4r_B$. It follows that when 
$y\in B(x_0,r_0)$, we must have $d(y,x_B)<r_0+4r_B\le 8r_B$, that is, $B(x_0,r_0)\subset 8B$. Then 
\[
\xi(B)\ge \frac{1}{D^3}\xi(8B)\ge \frac{1}{D^4}\xi(B(x_0,r_0)),
\]
and so
\[
\xi(B(x_0,r_0))\le D^4\, \xi(B)\le D^4\, \xi(\mathcal{B})<D^4\,\epsilon\, \xi(B(x_0,r_0)).
\]
This is not possible, by choosing $\epsilon$ small enough so that
\begin{equation}\label{eq:AA-Limit2}
	D^4\, \epsilon<1.
\end{equation}
As $x\not\in\bigcup\tau\mathcal{B}$, it follows now that for each $B\in\mathcal{B}$ we have $4B\cap\{x,x_0\}$ is empty. Hence
by Lemma~\ref{lem:ballpointsold}, there is some $z\in\{x,x_0\}$ and $0<r_x<8r_0$ such that 
\[
h_p(\mathcal{B}|_{B(z,r_x)})>\epsilon_0\, h_p(B(z,r_x)),
\]
and so by Lemma~\ref{lem:conentconv} we have that
\begin{equation}\label{eq:AA-5}
	\xi(\mathcal{B}|_{B(z,r_x)})>\frac{\epsilon_0}{K_1}\, \xi(B(z,r_x)).
\end{equation}
		
If $z=x_0$ and $\tfrac{r_0}{10}<r_x<8r_0$, then $\xi(B(z,r_x))=\xi(B(x_0,r_x))\ge D^{-4}\, \xi(B(x_0,r_0))$, 
and so by~\eqref{eq:AA-5},
\[
\frac{1}{D^4}\, \xi(B(x_0,r_0))\le \frac{K_1}{\epsilon_0}\, \xi(\mathcal{B}|_{B(x_0,r_x)})
\le \frac{K_1}{\epsilon_0}\, \xi(\mathcal{B})<\frac{K_1}{\epsilon_0}\, \epsilon\, \xi(B(x_0,r_0)),
\]
which will not be possible if we also ensure $\epsilon$ (or, equivalently, $\Delta$) is small enough so that
\begin{equation}\label{eq:AA-Limit3}
	\frac{D^4\, K_1}{\epsilon_0}\, \epsilon<1.
\end{equation}
On the other hand, if $z=x_0$ and $r_x\le \tfrac{r_0}{10}$, then $B(z,r_x)\subset B(x_0,r_0/10)$, and 
so if $B\in\mathcal{B}$ such that $B$ intersects
$B(z,r_x)$ and hence intersects $B(x_0,r_0/10)$, then by the virtue of $B$ containing a point in $\partial\Om$ we must have that
$2r_B>9r_0/10$, that is, $r_B>9r_0/20$. Now $d(x_B,x_0)<r_B+\tfrac{r_0}{10}<\tfrac{11}{9}r_B$. 
Therefore when $y\in B(x_0,r_0)$,
we must have $d(x_B,y)<\tfrac{11}{9}r_B+r_0<\tfrac{31}{9}r_B$. It follows that $B(x_0,r_0)\subset \tfrac{31}{9}B$, 
whence we obtain
\[
\xi(B)\ge D^{-4}\, \xi(\tfrac{31}{9}B)\ge D^{-5}\, \xi(B(x_0,r_0)).
\]
Therefore we have
\[
\frac{1}{D^5}\, \xi(B(x_0,r_0))\le \xi(B)\le \xi(\mathcal{B})<\epsilon\, \xi(B(x_0,r_0)),
\]
which is also not possible by making $\epsilon$ to be small enough so that
\begin{equation}\label{eq:AA-Limit4}
	D^5\, \epsilon<1.
\end{equation}
Therefore, for each $x\in B(x_0,2r_0)\setminus(\Om\cup\bigcup\tau\mathcal{B})$ we must have $z=x$, together with 
some $0<r_x<8r_0$ such that 
\begin{equation}\label{eq:AA-OnePigeon}
	\xi(\mathcal{B}|_{B(x,r_x)})\ge \frac{\epsilon_0}{K_1}\, \xi(B(x,r_x)).
\end{equation}
		
The collection $\{B(x,2r_x)\, :\, x\in B(x_0,2r_0)\setminus(\Omega\cup\bigcup\tau\mathcal{B})\}$ covers 
$B(x_0,2r_0)\setminus(\Omega\cup\bigcup\tau\mathcal{B})$. We apply the $5$-covering theorem to obtain a pairwise-disjoint
subcollection $\{B_i\}_{i\in I\subset\N}$ such that $\{5B_i\}_{i\in I}$ covers
 $B(x_0,2r_0)\setminus(\Omega\cup\bigcup\tau\mathcal{B})$. Since $\rad(B_i)\leq 8r_i$ for all $i\in I$, we get 
from~\eqref{eq:AA-4} and Lemma~\ref{lem:si-cont-scale-change} that
\begin{equation}\label{eq:AA-content}
\frac{3}{4}\, \delta\, \xi(B(x_0,r_0))<\mathcal{H}^\xi_{r_0}(B(x_0,2r_0)\setminus (\Om\cup\bigcup\tau\mathcal{B}))\leq C(40)\mathcal{H}^\xi_{40r_0}(B(x_0,2r_0)\setminus (\Om\cup\bigcup\tau\mathcal{B})) \leq \sum_{i\in I} \xi(5B_i).
\end{equation}
Now, if $B\in \mathcal{B}$ is such that $B\cap B(x,r_x)\ne\emptyset$ for some 
$x\in B(x_0,2r_0)\setminus(\Omega\cup\bigcup\tau\mathcal{B})$, then as $d(x,x_B)\ge \tau r_B$, we must have
$\tau r_B\le r_x+r_B$, and so necessarily
\[
r_x>(\tau-1)\, r_B.
\]
It follows that $B\subset B(x,2r_x)$. Consequently, if $B_i, B_j$ are distinct balls in our covering, $i,j\in I$, then 
every pair of balls $B\in \cB|_{\tfrac12}B_i$ and $B'\in \cB|_{\tfrac12}B_j$ 
satisfy $B\subset B_i$ and $B'\subset B_j$. 
Since $B_i\cap B_j = \emptyset$, we have that the collections $\cB|_{\tfrac12}B_i$ and 
$\cB|_{\tfrac12}B_j$ are disjoint. 
	
Now it follows from~\eqref{eq:AA-content} and the above disjointness property that 
\[
\frac{3\delta}{4C(40)}\, \xi(B(x_0,r_0))\le \sum_{i\in I}\xi(5B_i)\le D^4\, \sum_{i\in I}\xi(\tfrac12B_i)
\le D^4\, \frac{K_1}{\epsilon_0}\, \sum_{i\in I}\xi(\mathcal{B}|_{\tfrac12B_i})
\le D^4\, \frac{K_1}{\epsilon_0}\, \epsilon\, \xi(B(x_0,r_0)),
\]
which again is not possible by choosing $\Delta$ to be small enough so that
\begin{equation}\label{eq:AA-Limit5}
\frac{4D^{4}C(40)}{3}\, \frac{K_1}{\epsilon_0}\, \frac{\epsilon}{\delta}=\frac{4D^{4} C(40)}{3}\, \frac{K_1}{\epsilon_0}\, \Delta<1.
\end{equation}
These exhaust all the possibilities, and so
by ensuring that conditions~\eqref{eq:AA-Limit1}, \eqref{eq:AA-Limit2}, \eqref{eq:AA-Limit3}, 
\eqref{eq:AA-Limit4}, \eqref{eq:AA-Limit5}
all hold,~\eqref{eq:AA-1} and~\eqref{eq:AA-2} cannot simultaneously hold. This completes
the proof of Theorem~\ref{thm:mainthm-geodesic-bdry}.
\end{proof}

\bibliographystyle{acm}
\bibliography{pmodulus}

\begin{thebibliography}{10}

\bibitem{AMP04}
{\sc Ambrosio, L., Miranda, Jr., M., and Pallara, D.}
\newblock Special functions of bounded variation in doubling metric measure
  spaces.
\newblock In {\em Calculus of variations: topics from the mathematical heritage
  of {E}. {D}e {G}iorgi}, vol.~14 of {\em Quad. Mat.} Dept. Math., Seconda
  Univ. Napoli, Caserta, 2004, pp.~1--45.

\bibitem{azzam}
{\sc Azzam, J.}
\newblock Accessible parts of the boundary for domains with lower content
  regular complements.
\newblock {\em Ann. Acad. Sci. Fenn. Math. 44}, 2 (2019), 889--901.

\bibitem{BBS-hyp-fill}
{\sc Bj\"{o}rn, A., Bj\"{o}rn, J., and Shanmugalingam, N.}
\newblock Extension and trace results for doubling metric measure spaces and
  their hyperbolic fillings.
\newblock {\em J. Math. Pures Appl. (9) 159\/} (2022), 196--249.

\bibitem{ebkeithzhong}
{\sc Eriksson-Bique, S.}
\newblock Alternative proof of {K}eith-{Z}hong self-improvement and
  connectivity.
\newblock {\em Ann. Acad. Sci. Fenn. Math. 44}, 1 (2019), 407--425.

\bibitem{previouspaper}
{\sc Eriksson-Bique, S., Gibara, R., and Korte~R., Shanmugalingam, N.}
\newblock Traces of {N}ewton-{S}obolev functions on the visible boundary of
  domains in doubling metric measure spaces supporting a $p$-{P}oincar\'e
  inequality.
\newblock preprint, https://arxiv.org/abs/2308.09800, 2023.

\bibitem{FSS03}
{\sc Franchi, B., Serapioni, R., and Serra~Cassano, F.}
\newblock On the structure of finite perimeter sets in step 2 {C}arnot groups.
\newblock {\em J. Geom. Anal. 13}, 3 (2003), 421--466.

\bibitem{GibKort22}
{\sc Gibara, R., and Korte, R.}
\newblock Accessible parts of the boundary for domains in metric measure
  spaces.
\newblock {\em Ann. Fenn. Math. 47}, 2 (2022), 695--706.

\bibitem{GibKorSh}
{\sc Gibara, R., Korte, R., and Shanmugalingam, N.}
\newblock Solving a {D}irichlet problem on unbounded domains via a conformal
  transformation.
\newblock {\em Math. Ann.\/} (2023).

\bibitem{GibSh2}
{\sc Gibara, R., and Shanmugalingam, N.}
\newblock Trace and extension theorems for homogeneous {S}obolev and {B}esov
  spaces for unbounded uniform domains in metric measure spaces.
\newblock {\em Proceedings of the Steklov Mathematical Institute, Accepted: May
  2023, {\tt https://arxiv.org/abs/2211.12708}\/}.

\bibitem{hajkos95}
{\sc Haj{\l}asz, P., and Koskela, P.}
\newblock Sobolev meets {P}oincar\'{e}.
\newblock {\em C. R. Acad. Sci. Paris S\'{e}r. I Math. 320}, 10 (1995),
  1211--1215.

\bibitem{hajkos}
{\sc Haj{\l}asz, P., and Koskela, P.}
\newblock Sobolev met {P}oincar\'{e}.
\newblock {\em Mem. Amer. Math. Soc. 145}, 688 (2000).

\bibitem{HeiWall}
{\sc Heinonen, J.}
\newblock The boundary absolute continuity of quasiconformal mappings. {II}.
\newblock {\em Rev. Mat. Iberoamericana 12}, 3 (1996), 697--725.

\bibitem{heinonenkoskela}
{\sc Heinonen, J., and Koskela, P.}
\newblock Quasiconformal maps in metric spaces with controlled geometry.
\newblock {\em Acta Math. 181}, 1 (1998), 1--61.

\bibitem{HKST}
{\sc Heinonen, J., Koskela, P., Shanmugalingam, N., and Tyson, J.~T.}
\newblock {\em Sobolev spaces on metric measure spaces}, vol.~27 of {\em New
  Mathematical Monographs}.
\newblock Cambridge University Press, Cambridge, 2015.

\bibitem{HMM22}
{\sc Honzlov\'{a}-Exnerov\'{a}, V., Mal\'{y}, J., and Martio, O.}
\newblock {$AM$}-modulus and {H}ausdorff measure of codimension one in metric
  measure spaces.
\newblock {\em Math. Nachr. 295}, 1 (2022), 140--157.

\bibitem{ihlehrtuom}
{\sc Ihnatsyeva, L., Lehrb\"{a}ck, J., Tuominen, H., and V\"{a}h\"{a}kangas,
  A.~V.}
\newblock Fractional {H}ardy inequalities and visibility of the boundary.
\newblock {\em Studia Math. 224}, 1 (2014), 47--80.

\bibitem{KaenSuom07}
{\sc K\"{a}enm\"{a}ki, A., and Suomala, V.}
\newblock Conical upper density theorems and porosity of measures.
\newblock {\em Adv. Math. 217}, 3 (2008), 952--966.

\bibitem{Pugilistic}
{\sc Kinnunen, J., Korte, R., Shanmugalingam, N., and Tuominen, H.}
\newblock Lebesgue points and capacities via the boxing inequality in metric
  spaces.
\newblock {\em Indiana Univ. Math. J. 57}, 1 (2008), 401--430.

\bibitem{KosLehr09}
{\sc Koskela, P., and Lehrb\"{a}ck, J.}
\newblock Weighted pointwise {H}ardy inequalities.
\newblock {\em J. Lond. Math. Soc. (2) 79}, 3 (2009), 757--779.

\bibitem{KNN}
{\sc Koskela, P., Nandi, D., and Nicolau, A.}
\newblock Accessible parts of boundary for simply connected domains.
\newblock {\em Proc. Amer. Math. Soc. 146}, 8 (2018), 3403--3412.

\bibitem{Lah20-B}
{\sc Lahti, P.}
\newblock A new {F}ederer-type characterization of sets of finite perimeter.
\newblock {\em Arch. Ration. Mech. Anal. 236}, 2 (2020), 801--838.

\bibitem{Panu}
{\sc Lahti, P.}
\newblock Generalized {L}ipschitz numbers, fine differentiability, and
  quasiconformal mappings.
\newblock preprint, {\tt https://arxiv.org/abs/2202.05566}, 2022.

\bibitem{lahticapacity}
{\sc Lahti, P.}
\newblock Capacitary density and removable sets for {N}ewton-{S}obolev
  functions in metric spaces.
\newblock {\em Calc. Var. Partial Differential Equations 62}, 5 (2023), Paper
  No. 155, 20.

\bibitem{Lehr14}
{\sc Lehrb\"{a}ck, J.}
\newblock Weighted {H}ardy inequalities beyond {L}ipschitz domains.
\newblock {\em Proc. Amer. Math. Soc. 142}, 5 (2014), 1705--1715.

\bibitem{Maly}
{\sc Mal\'{y}, L.}
\newblock Trace and extension theorems for {S}obolev-type functions in metric
  spaces.
\newblock preprint, {\tt https://arxiv.org/abs/1704.06344}, 2017.

\bibitem{Mattila}
{\sc Mattila, P.}
\newblock {\em Geometry of sets and measures in {E}uclidean spaces}, vol.~44 of
  {\em Cambridge Studies in Advanced Mathematics}.
\newblock Cambridge University Press, Cambridge, 1995.
\newblock Fractals and rectifiability.

\bibitem{Suomala08}
{\sc Suomala, V.}
\newblock Upper porous measures on metric spaces.
\newblock {\em Illinois J. Math. 52}, 3 (2008), 967--980.

\bibitem{VWall-E}
{\sc V\"{a}is\"{a}l\"{a}, J.}
\newblock The wall conjecture on domains in {E}uclidean spaces.
\newblock {\em Manuscripta Math. 93}, 4 (1997), 515--534.

\bibitem{VWall}
{\sc V\"{a}is\"{a}l\"{a}, J.}
\newblock Wall properties of domains in {E}uclidean spaces.
\newblock {\em Ann. Acad. Sci. Fenn. Math. 27}, 2 (2002), 437--444.

\end{thebibliography}
\def\cprime{$'$}
	
\medskip
	
\noindent {\bf Address:} \\
	
\noindent S.E.-B.: University of Jyv\"askyl\"a, Department of Mathematics and Statistics, P.O. Box 35 (MaD), FI-40014 University of Jyv\"askyl\"a, Finland \\
\noindent E-mail: S.E-B.: {\tt sylvester.d.eriksson-bique@jyu.fi}\\
	
\vskip .2cm
	
\noindent R.G.: Department of Mathematical Sciences, P.O.~Box 210025, University of Cincinnati, Cincinnati, OH~45221-0025, U.S.A.\\
\noindent E-mail: R.G.: {{\tt gibararn@ucmail.uc.edu}\\ 
		
\vskip .2cm
		
\noindent R.K.: Department of Mathematics and Systems Analysis, Aalto University, P.O. Box 11100, FI-00076 Aalto, Finland.
		\\
\noindent E-mail:  R.K.: {\tt riikka.korte@aalto.fi}\\
		
\vskip .2cm
		
\noindent N.S.: Department of Mathematical Sciences, P.O.~Box 210025, University of Cincinnati, Cincinnati, OH~45221-0025, U.S.A.\\
\noindent E-mail:  N.S.: {\tt shanmun@uc.edu}\\	
		
\end{document}